\documentclass[12pt]{amsart}
\usepackage[english]{babel}
\usepackage{amsmath,amsthm}
\usepackage{amsfonts}
\usepackage{amssymb}
\usepackage{accents}
\usepackage{pstricks-add}
\usepackage{graphicx}
\usepackage{color}
\usepackage{framed}
\usepackage{mathptmx}
\usepackage{enumerate}
\usepackage[arrow, matrix, curve]{xy}
\usepackage{pstricks}
\usepackage{pstricks, pstricks-add, mathptmx, amsfonts, hyperref, color}
\usepackage[a4paper, left=2.5cm, right=2cm, top=2cm]{geometry}

\newtheorem{thm}{Theorem}[section]

\newtheorem{lem}[thm]{Lemma}
\newtheorem{prop}[thm]{Proposition}
\theoremstyle{definition}
\newtheorem{defn}[thm]{Definition}
\theoremstyle{remark}
\newtheorem{rem}[thm]{Remark}
\numberwithin{equation}{section}

\newcommand{\set}[1]{\left\{#1\right\}}
\newcommand{\Real}[0]{\mathbb R}
\newcommand{\Compl}[0]{\mathbb C}
\newcommand{\eps}[0]{\varepsilon}

\newcommand{\re}{\,\mathfrak{Re}\,}
\newcommand{\im}{\,\mathfrak{Im}\,}

\newcommand{\meas}{\,\text{meas}\,}
\newcommand{\tr}{\,\text{tr}\,}

\newcommand{\ii}{\mathrm{i}}
\newcommand{\sech}{\mathrm{sech}}

\newcommand{\co}{\text{const. }}
\def\res{\mathop{Res}}
\newcommand\sol[2]{\text{sol}^{#1}_{#2}(x,t)}

\begin{document}
\begin{abstract}
    In this article we consider the Cauchy problem for the cubic focusing nonlinear Schr\"{o}\-dinger (NLS) equation on the line with initial datum close to a particular $N$-soliton. Using inverse scattering and the $\overline{\partial}$ method we establish the decay of the $L^{\infty}(\Real)$ norm of the residual term in time.
\end{abstract}
\title{Asymptotic stability of $N$-solitons in the cubic NLS equation}
\address{Mathematisches Institut\\
  Universit\"{a}t zu K\"{o}ln\\
  50931 K\"{o}ln, Germany}
\email[A.~Saalmann]{asaalman@math.uni-koeln.de}
\author{Aaron Saalmann}
\date{\today}
\maketitle

\section{Introduction}\label{sec: intro}
We study the cubic focusing nonlinear Schr\"odinger (NLS) equation
\begin{equation}\label{equ: nls}
    \ii u_t+u_{xx}+2|u|^2u=0,
\end{equation}
on $\Real$ where $u(x,t):\Real\times\Real\to\Compl$. The initial value problem for (\ref{equ: nls}) is globally well posed in $L^2(\Real)$ due to the results of Tsutsumi \cite{Tsutsumi87}.
The \emph{linear} Schr\"{o}dinger equation $\ii q_t+q_{xx}=0$ is dispersive. Here, dispersion means that any solution $q$ of the linear Schr\"{o}dinger equation has the property $\|q\|_{L^{\infty}_x}\sim t^{-1/2}$ as $t\to\infty$. Once nonlinear effects are included \emph{soliton} solutions appear. Instead of dispersion, in the NLS equation we have that any solution decomposes into a solitary wave and a dispersive part as $t\to\infty$. \medskip \\ In this paper we will prove the following completion of Theorem 1.3 in \cite{Cuccagna2014}:
\begin{thm}\label{thm: mainthm}
    Fix $s\in (1/2,1]$, pairwise distinct poles $z_1,...,z_N\in\Compl^+$ and coupling constants $c_1,...,c_N\in\Compl\setminus\set{0}$ and denote the $N$-soliton with exactly these parameters by $u^{(sol)}$. Then we can find $\eps_0,T,C>0$ and solitons $u^{(sol)}_{\pm}$ with parameters $(z'_1,...,z'_N;c^{\pm}_1,..,c^{\pm}_N)$ with the following properties: for any $u_0\in L^{2,s}(\Real)\cap\mathcal{G}$ such that
    \begin{equation}\label{equ: u minus sol}
        \epsilon:=\|u_0(\cdot)-u^{(sol)}(\cdot,t=0)\|_{L^{2,s}(\Real)}<\eps_0,
    \end{equation}
    the solution of the initial value problem $u(\cdot,0)=u_0$ for (\ref{equ: nls}) satisfies
    \begin{equation}\label{equ: stability of n solitons}
        \left\| u(\cdot,t)-u^{(sol)}_{\pm}(\cdot,t)
            \right\|_{L^{\infty}(\Real)}<C\epsilon |t|^{-\frac12}
    \end{equation}
    for all $\pm t\geq T$. Additionally we have
    \begin{equation}\label{equ: z - hat z}
        |z_j-z'_j|+|c_j-c^{\pm}_j|<C\epsilon
    \end{equation}
    for all $1\leq j \leq N$.
\end{thm}
In the proof of Theorem \ref{thm: mainthm} we will compute the parameters of $u^{(sol)}_{\pm}$ explicitly and it will turn out that the $z'_j$ are given by the poles of $u_0$. The coupling constants $c^{\pm}_j$ can also be derived from the scattering data of $u_0$ via $c^{\pm}_j=c'_j(\Lambda_j^{\pm})^2$, where
\begin{equation}\label{equ: def Lambda}
   \Lambda_{j}^{\pm}:=\exp\left(\mp\frac{1}{2\pi\ii} \int_{\mp\infty}^{\re(z'_{j})} \frac{\log(1+|r(\varsigma))|^2} {\varsigma-z'_{j}}d\varsigma
   \right).
\end{equation}
In \cite{Pelinovsky2014} Contreras and Pelinovsky establish the orbital stability of $N$-solitons in the $L^2(\Real)$ space under the assumption (\ref{equ: u minus sol}). As mentioned by these authors they believed that also the (stronger) result (\ref{equ: stability of n solitons}) holds. For the proof we consider the Riemann Hilbert problem associated to the NLS equation and its solution $m$. Then motivated by the paper of  Cuccagna and Jenkins \cite{Jenkins2014} we define modifications $m\to m^{(1)}\to m^{(2)}\to m^{(3)}\to m^{(4)}\to m^{(5)}\to m^{(6)}\to m^{(7)}$ such that in the end $m^{(7)}$ is either trivial or corresponds to the $1$-soliton or to a breather solution. These modifications contain the Parabolic Cylinder RHP and the $\overline{\partial}$-method but not the dressing transformation like in \cite{Pelinovsky2014}.\medskip \\
The paper is organized as follows: Section \ref{sec: IST} gives some information about the direct in inverse scattering transform. Sections \ref{sec: poles} - \ref{sec: last step} are devoted to the chain of manipulations $m\to ...\to m^{(7)}$. Finally in Section \ref{sec: proof} the results will be collected in order to prove Theorem \ref{thm: mainthm}.\\
The results in \cite{Borghese2016} (and also the proofs) are basically the same but developed independently of each other.
\par \textbf{Acknowlegments.} I  wish to thank Prof. Scipio Cuccagna and Prof. Markus Kunze for useful discussions.

\section{The Inverse Scattering Transform and Soliton Solutions}\label{sec: IST}
A key ingredient for many results on stability of solitary waves comes from the methods of inverse scattering. The following theorem summarizes the theory:
\begin{thm}\label{thm: scattering}
    Let $s\in(1/2,1]$. There exist open sets $\mathcal{G}_N\subset L^1(\Real)$ ($N\in\mathbb{N}\cup\set{0}$) and transformations
    \begin{equation*}
        \mathcal{S}_N:\qquad
        \begin{aligned}
            L^{2,s}(\Real)\cap\mathcal{G}_N&\to &H^s(\Real)\times \Compl_+^N\times \Compl_{\ast}^N\\
            u_0 &\mapsto& (r(z);z_1,...,z_N;c_1,...,c_N)
        \end{aligned}
    \end{equation*}
    such that:
    \begin{enumerate}[(i)]
      \item $\mathcal{G}:=\bigcup_{N\in\mathbb{N}\cup\set{0}} \mathcal{G}_N$ is dense in $L^1$;
      \item The maps $\mathcal{S}_N$ are locally Lipschitz and one-to-one;
      \item The solution of (\ref{equ: nls}) with $u(x,0)=u_0(x)$ and $u_0\in H^1(\Real)\cap L^{2,s}(\Real)\cap\mathcal{G}_N$ can be obtained by the following three steps:
          \begin{description}
            \item[1. Step] Calculate the scattering data associated with $u_0$, i.e. $(r(z);z_1,...,z_N;c_1,...,c_N):= \mathcal{S}_N(u_0)$.
            \item[2. Step]Solve the following Riemann Hilbert problem:
\begin{samepage}
\begin{framed}
\textbf{RHP[NLS]:}\\Find for each $(x,t)\in\Real\times\Real$ a $2\times 2$-matrix valued function $\Compl\ni z\mapsto m(z;x,t)$ which satisfies
\begin{enumerate}[(i)]
  \item $m(z;x,t)$ is meromorphic in $\Compl\setminus\Real$ (with respect to the parameter $z$).
  \item $m(z;x,t)=1+\mathcal{O}\left(\frac{1}{z}\right)$ as $|z|\to\infty$.
  \item The non-tangential boundary values $m_{\pm}(z;x,t)$ exist for $z\in\Real$ and satisfy the jump relation $m_+=m_-V(r)$, where
      \begin{equation}\label{equ: def V(r)}
        V(z;x,t):=
        \left(
          \begin{array}{cc}
            1+|r(z)|^2 & e^{\overline{\phi}(z)}\overline{r}(z) \\
            e^{\phi(z)}r(z) & 1 \\
          \end{array}
        \right)
      \end{equation}
      with
      \begin{equation}\label{equ: phase}
         \phi(z):=2\ii xz+4\ii z^2t.
      \end{equation}
  \item $m$ has simple poles at $z_1,...,z_N,\overline{z}_1,...,\overline{z}_N$ with
      \begin{equation}\label{equ: Res orig 2}
        \begin{aligned}
          \res_{z=z_k}m(z;x,t)&=\lim_{z\to z_k}m(z;x,t)
          \left(
            \begin{array}{cc}
              0 & 0 \\
              c_k e^{\phi_k} & 0
            \end{array}
          \right),\\
          \res_{z=\overline{z_k}}m(z;x,t)&=\lim_{z\to \overline{z}_k}m(z;x,t)
          \left(
            \begin{array}{cc}
              0 & -\overline{c}_k e^{\overline{\phi}_k} \\
              0 & 0
            \end{array}
          \right).
        \end{aligned}
      \end{equation}
      Here we set
      \begin{equation}\label{equ: phi_k}
           \phi_k:=\phi(z_k)\qquad(k=1,...,N).
      \end{equation}
\end{enumerate}
\end{framed}
\end{samepage}
            \item[3. Step] Calculate the required solution via
                \begin{equation}\label{equ: rec}
                     u(x,t):=2\ii \lim_{z\to\infty}z\, \left[m(z;x,t)\right]_{12}.
                \end{equation}
          \end{description}
          Here $[\cdot]_{12}$ denotes the $1$-$2$-component of the matrix in the brackets.
    \end{enumerate}
\end{thm}
The fact that \textbf{RHP[NLS]} is uniquely solvable is pointed out by Deift and Park in \cite{Deift2011}. For the convenience of the reader we show roughly how the scattering maps $\mathcal{S}_N$ are defined:
Given a function $u(x)$ we set
\begin{equation*}
    P(z;x):=
    \left(
      \begin{array}{cc}
        -\ii z & u(x) \\
        -\overline{u}(x) & \ii z \\
      \end{array}
    \right)
\end{equation*}
and consider the ODE
\begin{equation}\label{equ: vx=Pv}
    v_x(z;x)=P(z;x)v(z;x).
\end{equation}
We define $\psi_j^{(\pm)}$ ($j=1,2$) to be the unique $\Compl^2$-valued solutions of (\ref{equ: vx=Pv}) with the boundary conditions
\begin{equation*}
    \lim_{x\to\pm\infty}\psi^{(\pm)}_j(z;x)e^{\pm\ii x z}= e_j,\quad j=1,2,
\end{equation*}
where $e_1=(1,0)^T$ and $e_2=(0,1)^T$. In general if $u(\cdot)\in L^1(\Real)$, the functions $\psi^{(-)}_1$ and $\psi^{(+)}_2$ exist for $\im z\geq 0$ whereas $\psi^{(+)}_1$ and $\psi^{(-)}_2$ exist for $\im z\leq 0$ (see \cite{Ablowitz2004}). In both cases the dependence on $z$ is analytic. Due to $\tr P=0$, expressions such as $\det[\psi^{(-)}_1|\psi^{(+)}_2]$ or $\det[\psi^{(+)}_1|\psi^{(-)}_1]$ do not depend on $x$. We set
\begin{eqnarray*}
  a(z) &:=& \det[\psi^{(-)}_1(z;x)|\psi^{(+)}_2(z;x)], \\
  b(z) &:=& \det[\psi^{(+)}_1(z;x)|\psi^{(-)}_1(z;x)],
\end{eqnarray*}
such that $a$ is defined for $z\in\Compl_+$ and $b$ is defined for $z\in\Real$. Additionally the map $z\mapsto a(z)$ is analytic in the upper plane $\Compl_+$. The sets $\mathcal{G}_N$ stated in Theorem \ref{thm: scattering} are now defined by the number of zeros of $a$:
\begin{equation*}
    \mathcal{G}_N:=\set{u\in L^1(\Real)|a \text{ admits exactly $N$ simple zeros }z_1,...,z_N\in\Compl_+}.
\end{equation*}
In \cite{Beals1984} Beals and Coifman show, that the $\mathcal{G}_N$ are indeed open. Furthermore they prove statement (i) of Theorem \ref{thm: scattering}.
Now we amount to the definition of the scattering data $(r(z);z_1,...,z_N;c_1,...,c_N)$:\par
\textbf{Reflection coefficient}: The so-called \emph{reflection coefficient} $r$ is given by
\begin{equation}\label{equ: reflection coeff}
           r(z):=\frac{b(z)}{a(z)},\qquad z\in\Real.
\end{equation}
As it is shown in \cite{Cuccagna2014} by Cuccagna, we have $r\in H^s(\Real)$ in the case of $u\in L^{2,s}(\Real)$. Note the analogy to the Fourier transform. See also \cite{Zhou1998} for more general results.\par
\textbf{Poles:} The $z_k$ are defined to be the simple zeros of $a$. Hence, we have $a(z_k)=0$ but $a'(z_k)\neq0$ (the $'$ indicates the derivative with respect to the complex parameter $z$). We will refer to them as \emph{poles} and we will denote the set $\set{z_1,...,z_N}$ by $\mathcal{Z}_+$. Furthermore we set $\mathcal{Z}_-:=\set{\overline{z}_1,...,\overline{z}_N}$ and $\mathcal{Z}:=\mathcal{Z}_+\cup\mathcal{Z}_-$ \par
\textbf{Norming constants:} The so-called \emph{norming constants} $c_1,..,c_N$ are given by $c_k:=\gamma_k/a'(z_k)$, where $\gamma_k$ are defined by the equations $\psi^{(-)}_1(z_k;x)=\gamma_k\psi^{(+)}_2(z_k;x)$. Due to $$\det[\psi^{(-)}_1(z_k;x)|\psi^{(+)}_2(z_k;x)]=a(z_k)=0$$ the two vectors $\psi^{(-)}_1(z_k;x)$ and $\psi^{(+)}_2(z_k;x)$ are indeed linearly dependent, which implies that the numbers $\gamma_k$ exist. They do not depend on $x$ which is verified by differentiation.\\
\par
Now we turn to the explanation of the second step, stated in Theorem \ref{thm: scattering} (iii). For $u\in\mathcal{G}_N$ it is an elementary calculation (see \cite{Ablowitz2004}) to show that
\begin{equation*}
    m(z;x):=\left\{
                \begin{aligned}
                  \left[\left. \frac{\psi^{(-)}_1(z;x)e^{\ii zx}} {a(z)}
                  \right|\psi^{(+)}_2(z;x)e^{-\ii zx}
                  \right], & \qquad
                  \hbox{if }z\in\Compl_+, \\
                  \left[\psi^{(+)}_1(z;x)e^{\ii zx}
                  \left|\frac{\psi^{(-)}_2(z;x)e^{-\ii zx}}
                  {\overline{a(\overline{z})}}
                  \right.\right], & \qquad
                  \hbox{if }z\in\Compl_-,
                \end{aligned}
              \right.
\end{equation*}
solves the following Riemann Hilbert problem:
\begin{samepage}
\begin{framed}
\textbf{RHP:}\\Find a $2\times 2$-matrix valued function $\Compl\ni z\mapsto m(z;x)$ which satisfies
\begin{enumerate}[(i)]
  \item $m(z;x)$ is meromorphic in $z$ on $\Compl\setminus\Real$.
  \item $m(z;x)=1+\mathcal{O}\left(\frac{1}{z}\right)$ as $|z|\to\infty$.
  \item The non-tangential boundary values $m_{\pm}(z;x)$ exist for $z\in\Real$ and satisfy the jump relation $m_+=m_-V$, where
      \begin{equation}\label{equ: V(z,x)}
        V(z;x)=
        \left(
          \begin{array}{cc}
            1+|r(z)|^2 & e^{-2\ii x z}\overline{r}(z) \\
            e^{2\ii xz}r(z) & 1 \\
          \end{array}
        \right).
      \end{equation}
  \item $m(z;x)$ has simple poles at $z_1,...,z_N,\overline{z}_1,...,\overline{z}_N$ with
      \begin{equation}\label{equ: res (x)}
        \begin{aligned}
          \res_{z=z_k}m(z)&=\lim_{z\to z_k}m(z)
          \left(
            \begin{array}{cc}
              0 & 0 \\
              c_k e^{2\ii xz_k} & 0
            \end{array}
          \right),\\
          \res_{z=\overline{z}_k}m(z)&=\lim_{z\to \overline{z}_k}m(z)
          \left(
            \begin{array}{cc}
              0 & -\overline{c}_k e^{-2\ii x \overline{z}_k} \\
              0 & 0
            \end{array}
          \right).
        \end{aligned}
      \end{equation}
\end{enumerate}
\end{framed}
\end{samepage}

From the differential equation (\ref{equ: vx=Pv}) one can obtain the asymptotic behavior of the functions $\psi_j^{(\pm)}(z;x)$ as $z\to\infty$. For instance we have (see page 25 in \cite{Ablowitz2004})
\begin{equation*}
    \psi^{(\pm)}_2(z;x)e^{-\ii zx}=
    \left(
      \begin{array}{c}
        0 \\
        1 \\
      \end{array}
    \right)+\frac{1}{2\ii z}
    \left(
      \begin{array}{c}
        u(x) \\
        \int_x^{\pm\infty}|u(y)|^2dy \\
      \end{array}
    \right)+\mathcal{O}\left(\frac{1}{z^2}\right),
\end{equation*}
which is equivalent to the following important formula:
\begin{equation*}
    u(x)=2\ii\lim_{z\to\infty}z [m(z;x)]_{12}.
\end{equation*}
Here $m(z;x)$ is the matrix defined from the functions $\psi_j^{\pm}$ as above. So far we have described  the forward scattering and the inverse scattering, since we can reconstruct the function $u$ from its scattering data.\\
Now we are going to take into account the time $t$. If $u$ also depends on $t$ (i.e. $u=u(x,t)$) and $u(\cdot,t)\in L^1(\Real)$ for any $t\in \Real$, we can obtain the functions $a$ and $b$ as above for all times $t\in \Real$. Thus, we have $a(z;t)$ and $b(z;t)$ and we can ask for the time evolution of these two functions. The miraculous fact is the following: if $u(x,t)$ solves the NLS equation (\ref{equ: nls}) and $u(\cdot,t)\in H^1(\Real)$ for all $t\in\Real$, then
\begin{equation*}
    \partial_t a(z;t) = 0 \quad\text{ and }\quad
  \partial_t b(z;t) = 4\ii z^2 b(z;t).
\end{equation*}
The derivation of these equations is based on the Lax pair representation of the NLS equation (see \cite{Deift1994}). Solving them for $a$ and $b$ we obtain $a(z;t)=a(z;0)$ and $r(z;t)=e^{4\ii z^2t}r(z;0)$. In particular, if at time $t=0$ the function $u(x,0)$ produces $N$ simple zeros $z_1,...,z_N$ of $z\mapsto a(z;0)$ and if $u$ evolves accordingly to the NLS equation, then $u(\cdot,t)$ will produce exactly the same $N$ simple zeros at any other time $t\in\Real$. In particular, the sets $\mathcal{G}_N$ are invariant under the flow of the NLS equation.\\
Since the poles $z_1,...,z_N$ remain unchanged over time, we can find by the same arguments as above the norming constants $c_k(t)$. They now depend on $t$ and their evolution is given by $c_k(t)=c_k(0)e^{4\ii z_k^2t}$. Altogether the scattering data of a function $u(\cdot,t)$, which is a solution of the NLS equation (\ref{equ: nls}), is given at time $t$ by
\begin{equation}\label{equ: evolution of scattering}
    (e^{4\ii z^2t}r(z);z_1,...,z_N;e^{4\ii z_1^2t}c_1,...,e^{4\ii z_N^2t}c_N),
\end{equation}
where $(r(z);z_1,...,z_N;c_1,...,c_N)$ are obtained from the initial data $u(x,0)=u_0(x)$. Inserting the time dependence into (\ref{equ: V(z,x)}) and (\ref{equ: res (x)}) we end up exactly with (\ref{equ: def V(r)}) and (\ref{equ: Res orig 2}). Summarized the method of (inverse) scattering works as follows:
\begin{equation}\label{tab: sheme of scattering}
\begin{xy}\xymatrixcolsep{5pc}
  \xymatrix{
    u_0\in L^{2,s}(\Real)\cap\mathcal{G}_N\ar@{|->}[d]_{\text{NLS equation (\ref{equ: nls})}} \ar@{|->}[r]^{\mathcal{S}_N} &  \ar@{|->}[d]_{\text{see (\ref{equ: evolution of scattering})}} (r(z);z_1,...,z_N;c_1,...,c_N)\\
    u(x,t)  &  \ar@{|->}[l]_-{\text{solve \textbf{RHP[NLS]}}} (e^{4\ii z^2t}r(z);z_1,...,z_N;e^{4\ii z_1^2t}c_1,...,e^{4\ii z_N^2t}c_N)
  }
\end{xy}
\end{equation}
We now give a definition of $N$-solitons in terms of the scattering data:
\begin{defn}
    A solution $u$ of (\ref{equ: nls}) is called \emph{$N$-soliton} or \emph{multi-soliton} if the initial datum $u_0$ belongs to $\mathcal{G}_N$ and the corresponding reflection coefficient vanishes ($r(z)\equiv 0$).
\end{defn}
If $u_0\in\mathcal{G}_0$ and $r\equiv 0$ the Riemann Hilbert problem \textbf{RHP[NLS]} (see Theorem \ref{thm: scattering}) then reduces to: (i) $m(z;x,t)$ is entire (with respect to $z$); (ii) $m(z;x,t)=1+\mathcal{O}(z^{-1})$ as $|z|\to\infty$. By Liouville's Theorem it follows that $m(z;x,t)\equiv 1$ and applying (\ref{equ: rec}) we obtain $u(x,t)\equiv 0$.\\
In the case of $N=1$ the ansatz
\begin{equation*}
    m(z;x,t)=1+\frac{A(x,t)}{z-z_1}+ \frac{\widetilde{A}(x,t)}{z-\overline{z}_1},
\end{equation*}
reduces \textbf{RHP[NLS]} to an algebraic system, which is solved by
\begin{equation}\label{equ: A and A tilde}
    \begin{aligned}
        A(x,t)=
        \left(
          \begin{array}{cc}
            \frac{2 \ii |c_1|^2e^{2\re(\phi_1(x,t))}\im(z_1)}{|c_1|^{2} e^{2\re(\phi_1(x,t))}+4 \im(z_1)^2} & 0 \\
            \im(z_1)e^{\ii[\arg(c_1)+\re(\phi_1(x,t))]}
        \sech\left[-\re(\phi_1(x,t))- \ln\left(\frac{|c_1|}{2\im(z_1)}\right)\right] & 0 \\
          \end{array}
        \right),\\
        \widetilde{A}(x,t)=
        \left(
         \begin{array}{cc}
            0 & -\im(z_1)e^{-\ii[\arg(c_1)+\re(\phi_1(x,t))]}
            \sech\left[-\re(\phi_1(x,t))- \ln\left(\frac{|c_1|}{2\im(z_1)}\right)\right] \\
            0 & \frac{-2 \ii |c_1|^2e^{2\re(\phi_1(x,t))}\im(z_1)}{|c_1|^{2} e^{2\re(\phi_1(x,t))}+4 \im(z_1)^2} \\
         \end{array}
        \right).
    \end{aligned}
\end{equation}
The explicit solution of the NLS equation, which can now be obtained by the reconstruction formula (\ref{equ: rec}), is commonly called \emph{soliton} or \emph{$1$-soliton}:
\begin{multline}\label{equ: 1-soliton}
        \sol{1}{z_1,c_1}:=
        -2\ii\im(z_1)e^{-\ii [ \arg(c_1)+2\re(z_1)x+ 4\re(z_1^2)t]}\\ \times
        \sech\left[2\im(z_1)(x+ 4\re(z_1)t)- \ln\left(\frac{|c_1|}{2\im(z_1)}\right)\right].
\end{multline}
It describes a single wave packet which is centered at
\begin{equation}\label{equ: center of sol}
    x_0=(2\im(z_1))^{-1}\ln\left(\frac{|c_1|}{2\im(z_1)}\right)
    -4\re(z_1)t.
\end{equation}
So we see, that the wave is propagating with the velocity $v=-4\re(z_1)$. In doing so, its envelope remains undistorted. Thus $\sol{1}{z_1,c_1}$ is indeed a soliton in the sense of the definition of Drazin and Johnson (see Section 1.2 in \cite{Drazin1989}). Multisolitons are not solitons in the sense of D. and J. but it can be shown that for $\re(z_j)\neq\re(z_k)$ ($j\neq k$) a $N$-soliton splits into $N$ individual $1$-solitons (see \cite{Zakharov1972}).

\section{Separating the Poles}\label{sec: poles}
The	quintessence of Lemmata \ref{lem: RHP1} and \ref{lem: m^1 approx m^sol} of this section we will be the following observation:
the set of those poles who will contribute to the solution $u(x,t)$ depends on the ratio $-x/(4t)$.\\
For the parameter
\begin{equation}\label{equ: def xi}
    \xi:=\frac{-x}{4t}
\end{equation}
we find
\begin{equation*}
    \re\phi(z;x,t) = 8\im(z) t(\xi-\re(z)).
\end{equation*}
and we conclude for $t>0$:
\begin{align*}
    \re\phi(z;x,t)>0,\quad&\text{ if }\left\{
    \begin{aligned}
        &&\im(z)>0\text{ and }\re(z)<\xi,\\
        &\text{or }&\im(z)<0\text{ and }\re(z)>\xi,
    \end{aligned}\right.
    \\
    \re\phi(z;x,t)<0,\quad&\text{ if }\left\{
    \begin{aligned}
        &&\im(z)>0\text{ and }\re(z)>\xi,\\
        &\text{or }&\im(z)<0\text{ and }\re(z)<\xi.
    \end{aligned}\right.
\end{align*}
For the $\phi_k$ defined in (\ref{equ: phi_k}) we have
\begin{equation}\label{equ: lim e^phi_k}
    \lim_{t\to\infty}|e^{\phi_k}|=
    \left\{
      \begin{array}{ll}
        0, & \hbox{if }\re z_k>\xi, \\
        \infty, & \hbox{if }\re z_k<\xi,
      \end{array}
    \right.
\end{equation}
and
\begin{equation*}
    |e^{\phi_k}|=1 \text{ if }\re z_k=\xi.
\end{equation*}
Hence for a fixed $\xi$ the poles $z_1,...,z_N$ are split in two classes. We set:
\begin{equation}\label{equ: def nabla and laplace}
    \begin{aligned}
        \bigtriangledown(\xi)&:=\set{k\in\set{1,...,N}|\re z_k<\xi},\\
        \bigtriangleup(\xi)&:=\set{k\in\set{1,...,N}|\re z_k\geq\xi}.
    \end{aligned}
\end{equation}
Since we do not exclude the case where two poles have the same real part, we have to label the poles in a new matter. We group the poles with respect to theirs real parts:
\begin{equation}\label{equ: new label of Z}
      \left\{
        \begin{array}{ll}
          \mathcal{Z}_+=\set{z_1,...,z_N}=\set{z^{(1)}_1,...,z^{(1)}_{m_1},\;
       z^{(2)}_1,...,z^{(2)}_{m_2},\;
       ...\quad...,
       z^{(K)}_1,...,z^{(K)}_{m_K}},\\
          m_l\geq 1,\quad
    \sum_{l=1}^K m_l =N,\\
        \re z^{(l)}_j=\re z^{(p)}_h\quad\Leftrightarrow\quad l=p.
       \end{array}
      \right.
\end{equation}
For $t$ sufficiently large the set
\begin{equation}\label{equ: def square}
    \square(\xi):=\set{z\in\mathcal{Z}\left|\phantom{|^i}\right. \!\!\!\!|\re(z)-\xi|\leq 1/\sqrt{t}}
\end{equation}
depends only on $\xi$ and is either empty or equals exactly $\set{z^{(l)}_1,...,z^{(l)}_{m_l}, \overline{z}^{(l)}_1,...,\overline{z}^{(l)}_{m_l}}$ for one certain $l$. Now we define the contour
\begin{equation}\label{equ: def Sigma}
    \Sigma^{(1)}(x,t):=\bigcup_{\substack{z\in\mathcal{Z}\\z\notin \square(\xi)}} \partial B_{1/\sqrt{t}}(z),
\end{equation}
Next we set
\begin{equation}\label{equ: def T(z)}
    T(z;x,t):=\prod_{k\in\bigtriangledown(\xi)} \frac{z-z_k}{z-\overline{z}_k},
\end{equation}
and
\begin{equation}\label{equ: def D(z)}
    D(z;x,t):=T(z;x,t)^{\sigma_3}:=
    \left(
      \begin{array}{cc}
        T(z;x,t) & 0 \\
        0 & T(z;x,t)^{-1}\\
      \end{array}
    \right),
\end{equation}
such that we are now in a position to formulate the first modification of \textbf{RHP[NLS]}. From now on we will often drop the dependence on $x$ and $t$. For $m:\Compl\to\Compl^{2\times 2}$ we set
\begin{equation}\label{equ: def m^1}
    m^{(1)}(z):=
    \left\{
      \begin{array}{ll}\vspace{.1cm}
        m(z)\left(
              \begin{array}{cc}
                1 & -\frac{z-z_k}{c_ke^{\phi_k}} \\
                0 & 1 \\
              \end{array}
            \right)
        D(z), & \hbox{if }z\in B_{1/\sqrt{t}}(z_k), k\in\bigtriangledown(\xi), z_k\notin \square(\xi), \\ \vspace{.1cm}
        m(z)\left(
              \begin{array}{cc}
                1 & 0 \\
                -\frac{c_ke^{\phi_k}}{z-z_k} & 1 \\
              \end{array}
            \right)
        D(z), & \hbox{if }z\in B_{1/\sqrt{t}}(z_k), k\in\bigtriangleup(\xi), z_k\notin \square(\xi), \\ \vspace{.1cm}
        m(z)\left(
              \begin{array}{cc}
                1 & 0\\
                \frac{z-\overline{z}_k} {\overline{c}_ke^{\overline{\phi}_k}} & 1 \\
              \end{array}
            \right)
        D(z), & \hbox{if }z\in B_{1/\sqrt{t}}(\overline{z}_k), k\in\bigtriangledown(\xi), \overline{z}_k\notin \square(\xi), \\ \vspace{.1cm}
        m(z)\left(
              \begin{array}{cc}
                1 & \frac{\overline{c}_ke^{\overline{\phi}_k}} {z-\overline{z}_k} \\
                0 & 1 \\
              \end{array}
            \right)
        D(z), & \hbox{if }z\in B_{1/\sqrt{t}}(\overline{z}_k), k\in\bigtriangleup(\xi), \overline{z}_k\notin \square(\xi), \\
        m(z)D(z), & \hbox{else.}
      \end{array}
    \right.
\end{equation}
\begin{samepage}
\begin{lem}\label{lem: RHP1}
    If $m(z)$ solves \textbf{RHP[NLS]}, then $m^{(1)}(z)$ defined in (\ref{equ: def m^1}) is a solution to the following RHP:
    \begin{framed}
    \textbf{RHP[1]:}
    \begin{enumerate}[(i)]
      \item $m^{(1)}(z)$ is meromorphic in $\Compl\setminus (\Sigma^{(1)}\cup \Real)$.
      \item $m^{(1)}(z)=1+\mathcal{O}\left(\frac{1}{z}\right)$ as $|z|\to\infty$.
      \item If $\square(\xi)=\varnothing$, $m^{(1)}$ has no poles (i.e. $m^{(1)}$ is analytic on $\Compl\setminus (\Sigma^{(1)}\cup \Real)$). If $\square(\xi)$ consists of certain $z_k$ and $\overline{z}_{k}$ such that $k\in\bigtriangledown(\xi)$, $m^{(1)}$ has simple poles at these $z_k$ and $\overline{z}_{k}$ with:
          \begin{equation}\label{equ: Res 1 j in nabla}
            \begin{aligned}
              \res_{z=z_k}m^{(1)}(z)&=\lim_{z\to z_k}m^{(1)}(z)
              \left(
                \begin{array}{cc}
                  0 & \frac{1}{c_k e^{\phi_k}(T'(z_k))^2} \\
                  0 & 0
                \end{array}
              \right),\\
              \res_{z=\overline{z}_k}m^{(1)}(z)&=\lim_{z\to \overline{z}_k}m^{(1)}(z)
              \left(
                \begin{array}{cc}
                  0 & 0 \\
                  \frac{-1}{\overline{c}_k e^{\overline{\phi}_k} (\overline{T'(z_k)})^2} & 0
                \end{array}
              \right).
            \end{aligned}
          \end{equation}
          If $\square(\xi)$ consists of certain $z_k$ and $\overline{z}_{k}$ such that $k\in\bigtriangleup(\xi)$, $m^{(1)}$ has simple poles at these $z_k$ and $\overline{z}_{k}$ with:
          \begin{equation}\label{equ: Res 1 j in laplace}
            \begin{aligned}
              \res_{z=z_{k}}m^{(1)}(z)&=\lim_{z\to z_{k}}m^{(1)}(z)
              \left(
                \begin{array}{cc}
                  0 & 0 \\
                  c_{k} e^{\phi_{k}}(T(z_{k}))^2& 0
                \end{array}
              \right),\\
              \res_{z=\overline{z}_{k}}m^{(1)}(z)&=\lim_{z\to \overline{z}_{k}}m^{(1)}(z)
              \left(
                \begin{array}{cc}
                  0 & -\overline{c}_{k} e^{\overline{\phi}_{k}} (\overline{T(z_{k})})^2 \\
                  0 & 0
                \end{array}
              \right).
            \end{aligned}
          \end{equation}
      \item The non-tangential boundary values $m_{\pm}^{(1)}(z)$ exist for $z\in\Sigma^{(1)}\cup\Real$ and satisfy the jump relation $m_+^{(1)}=m_-^{(1)}V^{(1)}$, where
          \begin{equation}\label{equ: def V^1}
            V^{(1)}(z)=
              \left\{
                \begin{array}{ll}\vspace{.1cm}
                  \left(
                      \begin{array}{cc}
                       1 & \frac{z-z_k}{c_ke^{\phi_k}(T(z))^2} \\
                       0 & 1 \\
                      \end{array}
                  \right), & \hbox{if } z\in \partial B_{1/\sqrt{t}}(z_k), k\in\bigtriangledown(\xi), z_k\notin\square(\xi),\\ \vspace{.1cm}
                  \left(
                      \begin{array}{cc}
                       1 & 0 \\
                       \frac{c_ke^{\phi_k}(T(z))^2}{z-z_k} & 1 \\
                      \end{array}
                  \right), & \hbox{if }z\in \partial B_{1/\sqrt{t}}(z_k),k\in\bigtriangleup(\xi), z_k\notin\square(\xi),  \\ \vspace{.1cm}
                  \left(
                      \begin{array}{cc}
                       1 & 0 \\
                       -\frac{(z-\overline{z}_k)(T(z))^2} {\overline{c}_ke^{\overline{\phi}_k}} & 1 \\
                      \end{array}
                  \right), & \hbox{if }z\in \partial B_{1/\sqrt{t}}(\overline{z}_k), k\in\bigtriangledown(\xi), \overline{z}_k\notin\square(\xi), \\ \vspace{.1cm}
                  \left(
                      \begin{array}{cc}
                       1 & \frac{-\overline{c}_k e^{\overline{\phi}_k}} {(z-\overline{z}_k)(T(z))^2} \\
                       0& 1 \\
                      \end{array}
                  \right), & \hbox{if }z\in \partial B_{1/\sqrt{t}}(\overline{z}_k), k\in\bigtriangleup(\xi), \overline{z}_k\notin\square(\xi),\\
                  D^{-1}(z)V(z)D(z), & \hbox{if }z\in\Real.
                \end{array}
              \right.
          \end{equation}
    \end{enumerate}
    \end{framed}
\end{lem}
\end{samepage}
\begin{proof}
    ($i$) is trivial, ($ii$) is a consequence of
    \begin{equation*}
        D(z)=1+\mathcal{O}\left(\frac{1}{z}\right) \quad\text{as } |z|\to\infty.
    \end{equation*}
    ($iv$) is also elementary. It remains to show, that ($iii$) holds. We have therefore to show that
    (\ref{equ: Res 1 j in nabla}) and
    (\ref{equ: Res 1 j in laplace}) are correct and moreover we have to show that the poles at $z_k$ and $\overline{z}_k$ are indeed removed in the case of $z_k,\overline{z}_k\notin\square(\xi)$. Firstly we consider $m^{(1)}$ close to $z_k$ in the case where $z_k\notin\square(\xi)$ and $k \in \bigtriangledown(\xi)$: Let $m$ be a solution of \textbf{RHP[NLS]}. Then we have
    \begin{equation*}
        m(z)=\frac{A_k}{z-z_k}+B_k+\mathcal{O}(|z-z_k|) \quad(\text{as }z\to z_k)
    \end{equation*}
    with suitable matrices $A_k=A_k(x,t)$ and $B_k=B_k(x,t)$. The residua conditions in \textbf{RHP[NLS]} then yield
    the following two relations:
    \begin{equation}\label{equ: A_k v = 0}
        A_k
          \left(
            \begin{array}{cc}
              0 & 0 \\
              c_k e^{\phi_k} & 0
            \end{array}
          \right)=0,
    \end{equation}
    \begin{equation}\label{equ: A_k = B_k v}
        A_k=B_k
        \left(
            \begin{array}{cc}
              0 & 0 \\
              c_k e^{\phi_k} & 0
            \end{array}
          \right).
    \end{equation}
    By definition, (\ref{equ: A_k v = 0}) and
    (\ref{equ: A_k = B_k v}) we get
    \begin{eqnarray*}
      m^{(1)}(z) &=& m(z)
      \left(
        \begin{array}{cc}
          1 & -\frac{z-z_k}{c_ke^{\phi_k}} \\
          0 & 1 \\
        \end{array}
      \right)
        D(z)\\
       &=&\left[\frac{A_k}{z-z_k}+B_k+ \mathcal{O}(|z-z_k|)\right]
       \left[1+
       \left(
        \begin{array}{cc}
          0 & \frac{-1}{c_ke^{\phi_k}} \\
          0 & 0 \\
        \end{array}
       \right)(z-z_k)
       \right]\\&&\qquad
       \left[\frac{
       \left(
         \begin{array}{cc}
           0 & 0 \\
           0 & \frac{1}{T'(z_k)} \\
         \end{array}
       \right)}
       {z-z_k}+
       \left(
         \begin{array}{cc}
           0 & 0 \\
           0 & \ast \\
         \end{array}
       \right)
       + \mathcal{O}(|z-z_k|)\right] \\
       &=&\mathcal{O}(1)
    \end{eqnarray*}
    and it follows that there is no pole at $z_k$. In the case of $k\in\bigtriangleup(\xi)$ we find:
    \begin{eqnarray*}
      m(z)\left(
              \begin{array}{cc}
                1 & 0 \\
                -\frac{c_ke^{\phi_k}}{z-z_k} & 1 \\
              \end{array}
            \right)
       &=&  \left[\frac{A_k}{z-z_k}+B_k+ \mathcal{O}(|z-z_k|)\right]
       \left[1+\frac{
       \left(
        \begin{array}{cc}
          0 & 0 \\
          -c_ke^{\phi_k} & 0 \\
        \end{array}
       \right)}{z-z_k}
       \right]\\
       &=&\frac{A_k
       \left(
        \begin{array}{cc}
          1 & 0 \\
          -c_ke^{\phi_k} & 1 \\
        \end{array}
       \right)}{(z-z_k)^2}+\frac{
       B_k
       \left(
        \begin{array}{cc}
          0 & 0 \\
          -c_ke^{\phi_k} & 0 \\
        \end{array}
       \right)+A_k
       }{z-z_k}+\mathcal{O}(1)\\
    &\stackrel{
    (\ref{equ: A_k v = 0})\&
    (\ref{equ: A_k = B_k v})}{=}&\mathcal{O}(1)
    \end{eqnarray*}
    Since $D(z)$ has no pole at $z_k$ ($k\in\bigtriangleup(\xi)$), it is clear that also $m^{(1)}(z)=\mathcal{O}(1)$ as $z\to z_k$.\\
    The calculations for $\overline{z}_k\notin\square(\xi)$ ($k\in(\bigtriangledown(\xi)\cup\bigtriangleup(\xi))$) are similar. Now we turn to establish the first line of
    (\ref{equ: Res 1 j in nabla}): Let us assume $z_k\in \square(\xi)$ and $k\in\bigtriangledown(\xi)$.
    We use
    \begin{equation*}
        m(z)=\frac{A_k}{z-z_k}+B_k+ C_k(z-z_k)+\mathcal{O}(|z-z_k|^2)
    \end{equation*}
    and
    \begin{equation*}
        D(z)=\frac{
       \left(
         \begin{array}{cc}
           0 & 0 \\
           0 & \frac{1}{T'(z_k)} \\
         \end{array}
       \right)}
       {z-z_k}+
       \left(
         \begin{array}{cc}
           0 & 0 \\
           0 & \ast \\
         \end{array}
       \right)+
       \left(
         \begin{array}{cc}
           T'(z_k) & 0 \\
           0 & \ast \\
         \end{array}
       \right)(z-z_k)+\mathcal{O}(|z-z_k|^2)
    \end{equation*}
    to obtain for $z$ close to $z_k$
    \begin{eqnarray*}
      m^{(1)}(z)
       &\stackrel{(\ref{equ: A_k v = 0})}{=}&
       \frac{B_k
       \left(
         \begin{array}{cc}
           0 & 0 \\
           0 & \frac{1}{T'(z_k)} \\
         \end{array}
       \right)
       }{z-z_k}+
       A_k
       \left(
         \begin{array}{cc}
           T'(z_k) & 0 \\
           0 & \ast \\
         \end{array}
       \right)+B_k
       \left(
         \begin{array}{cc}
           0 & 0 \\
           0 & \ast \\
         \end{array}
       \right)+C_k
       \left(
         \begin{array}{cc}
           0 & 0 \\
           0 & \frac{1}{T'(z_k)} \\
         \end{array}
       \right)\\&&\quad+\mathcal{O}(|z-z_k|).
    \end{eqnarray*}
    On the one hand, from this expansion we find
    \begin{equation}\label{equ: Res m^1}
        \res_{z=z_k}m^{(1)}(z)=
        B_k
       \left(
         \begin{array}{cc}
           0 & 0 \\
           0 & \frac{1}{T'(z_k)} \\
         \end{array}
       \right)
    \end{equation}
    and on the other hand
    \begin{equation}\label{equ: lim m^1}
        \begin{aligned}
        &\lim_{z\to z_k}m^{(1)}(z)
              \left(
                \begin{array}{cc}
                  0 & \frac{1}{c_k e^{\phi_k}(T'(z_k))^2} \\
                  0 & 0
                \end{array}
              \right)
            \\&=\left[
              A_k
       \left(
         \begin{array}{cc}
           T'(z_k) & 0 \\
           0 & \ast \\
         \end{array}
       \right)+B_k
       \left(
         \begin{array}{cc}
           0 & 0 \\
           0 & \ast \\
         \end{array}
       \right)+C_k
       \left(
         \begin{array}{cc}
           0 & 0 \\
           0 & \frac{1}{T'(z_k)} \\
         \end{array}
       \right)
              \right]
              \left(
                \begin{array}{cc}
                  0 & \frac{1}{c_k e^{\phi_k}(T'(z_k))^2} \\
                  0 & 0
                \end{array}
              \right)\\
        &=A_k \left(
                \begin{array}{cc}
                  0 & \frac{1}{c_k e^{\phi_k}T'(z_k)} \\
                  0 & 0
                \end{array}
              \right)\\
         &= B_k
       \left(
         \begin{array}{cc}
           0 & 0 \\
           0 & \frac{1}{T'(z_k)} \\
         \end{array}
       \right).
        \end{aligned}
    \end{equation}
    (\ref{equ: Res m^1}) and (\ref{equ: lim m^1}) prove the first line of (\ref{equ: Res 1 j in nabla}). The second line follows from analog calculations. Alternatively we can say that the first line of (\ref{equ: Res 1 j in nabla}) implies the second since $m^{(1)}$ obeys the symmetry
    \begin{equation}\label{equ: symmetry of m^1}
        \overline{m^{(1)}(z)}=\sigma_2 m^{(1)}(\overline{z})\sigma_2,
    \end{equation}
    which can be derived from the symmetries $\overline{m(z)}=\sigma_2 m(\overline{z})\sigma_2$
    and $\overline{D(z)}=\sigma_2 D(\overline{z})\sigma_2$. Now we prove (\ref{equ: Res 1 j in laplace}). Let $z_k\in \square(\xi)$ and $k\in\bigtriangleup(\xi)$.
    \begin{eqnarray*}
      \res_{z=z_k}m^{(1)}(z)
        &=& \left[\res_{z=z_k}m(z)\right]D(z_k) \\
        &=& \lim_{z\to z_k}\left[m(z)
            \left(
                \begin{array}{cc}
                  0 & 0 \\
                  c_k e^{\phi_k}& 0
                \end{array}
            \right)\right]D(z_k)\\
        &=& \lim_{z\to z_k}\left[m^{(1)}(z)D(z_k)^{-1}
              \left(
                \begin{array}{cc}
                  0 & 0 \\
                  c_k e^{\phi_k}& 0
                \end{array}
              \right)\right]    D(z_k)\\
        &=& \lim_{z\to z_k}\left[m^{(1)}(z)
              \left(
                \begin{array}{cc}
                  0 & 0 \\
                  c_k e^{\phi_k}(T(z_k))^2& 0
                \end{array}
              \right)\right]
    \end{eqnarray*}
    \begin{eqnarray*}
      \res_{z=\overline{z}_k}m^{(1)}(z)
        &=& \left[\res_{z=\overline{z}_k}m(z)\right] D(\overline{z}_k) \\
        &=& \lim_{z\to \overline{z}_k} \left[m(z)
            \left(
                \begin{array}{cc}
                  0 & -\overline{c}_k e^{\overline{\phi}_k} \\
                  0& 0
                \end{array}
            \right)\right]D(\overline{z}_k)\\
        &=& \lim_{z\to \overline{z}_k} \left[m^{(1)}(z)D(\overline{z}_k)^{-1}
              \left(
                \begin{array}{cc}
                  0 & -\overline{c}_k e^{\overline{\phi}_k} \\
                  0& 0
                \end{array}
            \right)\right]    D(\overline{z}_k)\\
        &=& \lim_{z\to \overline{z}_k}\left[m^{(1)}(z)
              \left(
                \begin{array}{cc}
                  0 & \frac{-\overline{c}_k e^{\overline{\phi}_k}} {(T(\overline{z}_k))^2} \\
                  0& 0
                \end{array}
            \right)\right] \\
        &=&\lim_{z\to \overline{z_k}}\left[m^{(1)}(z)
              \left(
                \begin{array}{cc}
                  0 & -\overline{c}_k e^{\overline{\phi}_k} (\overline{T(z_k)})^2 \\
                  0 & 0
                \end{array}
              \right)\right]
    \end{eqnarray*}
    The last step is possible, because of the symmetry $\overline{T(z)}=\frac{1}{T(\overline{z})}$.
\end{proof}
We have used the function $T(z,x,t)$ to define the transformation $m\mapsto m^{(1)}$. As a consequence the poles at $z_k$ (and $\overline{z}_k$, respectively) are removed and instead a jump on the correspondent disk boundaries appears. Next we are going to prove rigorously the fact that this jump $V^{(1)}$ on $\Sigma^{(1)}$ defined in (\ref{equ: def V^1}) does not meaningfully contribute to the solution of \textbf{RHP[1]} as $t\to\infty$. Therefore we consider again a Riemann Hilbert problem:
\begin{samepage}
\begin{framed}
\textbf{RHP[2]:}\\Find a $2\times 2$-matrix valued function $\Compl\ni z\mapsto m^{(2)}(z)$ which satisfies
\begin{enumerate}[(i)]
  \item $m^{(2)}(z)$ is meromorphic in $\Compl\setminus\Real$,
  \item $m^{(2)}(z)=1+ \mathcal{O}\left(\frac{1}{z}\right)$ as $|z|\to\infty$,
  \item If $\square(\xi)=\varnothing$, $m^{(2)}$ has no poles (i.e. $m^{(2)}$ is analytic in $\Compl\setminus\Real$). If $\square(\xi)$ consists of certain $z_k$ and $\overline{z}_{k}$ such that $k\in\bigtriangledown(\xi)$, $m^{(2)}$ has simple poles at these $z_k$ and $\overline{z}_{k}$ with:
      \begin{equation}\label{equ: Res sol j in nabla}
            \begin{aligned}
              \res_{z=z_k}m^{(2)}(z)&=\lim_{z\to z_k}m^{(2)}(z)
              \left(
                \begin{array}{cc}
                  0 & \frac{1}{c_k e^{\phi_k}(T'(z_k))^2} \\
                  0 & 0
                \end{array}
              \right),\\
              \res_{z=\overline{z}_k}m^{(2)}(z) &=\lim_{z\to \overline{z}_k}m^{(2)}(z)
              \left(
                \begin{array}{cc}
                  0 & 0 \\
                  \frac{-1}{\overline{c}_k e^{\overline{\phi}_k} (\overline{T'(z_k)})^2} & 0
                \end{array}
              \right).
            \end{aligned}
      \end{equation}
      If $\square(\xi)$ consists of certain $z_k$ and $\overline{z}_{k}$ such that $k\in\bigtriangleup(\xi)$, $m^{(2)}$ has simple poles at these $z_k$ and $\overline{z}_{k}$ with:
      \begin{equation}\label{equ: Res sol j in laplace}
            \begin{aligned}
              \res_{z=z_k}m^{(2)}(z)&=\lim_{z\to z_k}m^{(2)}(z)
              \left(
                \begin{array}{cc}
                  0 & 0 \\
                  c_k e^{\phi_k}(T(z_k))^2& 0
                \end{array}
              \right),\\
              \res_{z=\overline{z}_k}m^{(2)}(z) &=\lim_{z\to \overline{z}_k}m^{(2)}(z)
              \left(
                \begin{array}{cc}
                  0 & -\overline{c}_k e^{\overline{\phi}_k} (\overline{T(z_k)})^2 \\
                  0 & 0
                \end{array}
              \right).
            \end{aligned}
      \end{equation}
      \item The non-tangential boundary values $m_{\pm}^{(2)}(z)$ exist for $z\in\Real$ and satisfy the jump relation $m_+^{(2)}=m_-^{(2)}V^{(2)}$, where
          \begin{equation}\label{equ: def V^2}
            V^{(2)}(z)=
                  D^{-1}(z)V(z)D(z)
          \end{equation}
\end{enumerate}
\end{framed}
\end{samepage}
\textbf{RHP[2]} can be viewed as \textbf{RHP[1]} with $V^{(1)}|_{\Sigma^{(1)}}\equiv 1$. Since $\lim_{t\to\infty}V^{(1)}(z)=1$ for $z\in\Sigma^{(1)}$, due to (\ref{equ: lim e^phi_k}), it is not surprising that somehow the solution of \textbf{RHP[1]} is converging to that of \textbf{RHP[2]} as $t\to\infty$. Indeed, we have:
\begin{lem}\label{lem: m^1 approx m^sol}
    There is a matrix $C_1(x,t)$ for which
    \begin{equation*}
        \|C_1\|\leq c e^{-8\sqrt{t}}\qquad(t>0)
    \end{equation*}
    (with $c>0$ independent of $x$) holds and such that
    \begin{equation}\label{equ: m^1 approx m^sol}
         m^{(1)}(z)= \left[ 1+\frac{C_1}{z}+\mathcal{O}\left(\frac{1}{z^2}\right) \right]m^{(2)}(z)
    \end{equation}
    as $|z|\to\infty$. As indicated by the notation, here $m^{(1)}$ solves \textbf{RHP[1]} and $m^{(2)}$ is a solution to \textbf{RHP[2]}, respectively.
\end{lem}
\begin{proof}
    We claim, that in each of the two cases $\square(\xi)=\varnothing$ and $\square(\xi)\neq\varnothing$ the matrix valued function $C(z):= m^{(1)}(z) \left[m^{(2)}(z)\right]^{-1}$ is a solution to
    \begin{samepage}
    \begin{framed}
        \textbf{RHP[C]}
        \begin{enumerate}[(i)]
          \item $C$ is analytic  in $\Compl\setminus\Sigma^{(1)}$.
          \item $C(z)= 1+ \mathcal{O}\left(\frac{1}{z}\right)$ as $|z|\to\infty$.
          \item The non-tangential boundary values $C_{\pm}(z)$ exist for $z\in\Sigma^{(1)}$ and satisfy the jump relation $C_+=C_-V^{(C)}$, where
              \begin{equation*}
                V^{(C)}(z)=m^{(2)}(z)\:V^{(1)}\big| _{\Sigma^{(1)}}(z) \left[m^{(2)}(z)\right]^{-1}.
              \end{equation*}
        \end{enumerate}
    \end{framed}
    \end{samepage}
    In order to prove (i), we have to show that
    \begin{equation*}
        C(z)=\mathcal{O}(1)\qquad\text{as }z\to z_k,\overline{z}_k
    \end{equation*}
    (if $z_k,\overline{z}_k\in\square(\xi)$). We begin with $k\in\bigtriangledown$ and consider $C(z)$ close to $z_k$. By $\det m^{(1)} \equiv \det m^{(2)} \equiv 1 $,   (\ref{equ: Res 1 j in nabla}) and                 (\ref{equ: Res sol j in nabla})
    (see also (\ref{equ: A_k v = 0}) and (\ref{equ: A_k = B_k v})) we have:
    \begin{eqnarray*}
      m^{(1)}(z) &=&\frac{
      \left(
        \begin{array}{cc}
          0 & \alpha \\
          0 & \beta \\
        \end{array}
      \right)}{z-z_k}+
      \left(
        \begin{array}{cc}
          \alpha/\eta_k & \ast_{12} \\
          \beta/\eta_k & \ast_{22} \\
        \end{array}
      \right)+\mathcal{O}(|z-z_k|)
      \\
      \left[m^{(2)}(z)\right]^{-1} &=& \frac{
      \left(
        \begin{array}{cc}
          \widetilde{\beta} & -\widetilde{\alpha} \\
          0 & 0 \\
        \end{array}
      \right)}{z-z_k}+
      \left(
        \begin{array}{cc}
          \widetilde{\ast}_{22} & \widetilde{\ast}_{12}\\
          -\widetilde{\beta}/\eta_k & \widetilde{\alpha}/\eta_k \\
        \end{array}
      \right)+\mathcal{O}(|z-z_k|)
    \end{eqnarray*}
    with suitable numbers $\alpha,\beta,\widetilde{\alpha},\widetilde{\beta}$ and $\eta_k:=\frac{1}{c_k e^{\phi_k}(T'(z_k))^2}$.
    After multiplication we arrive at
    \begin{eqnarray*}
      C(z)
       &=& \frac{
       \left(
        \begin{array}{cc}
          \alpha\widetilde{\beta}/\eta_k & -\alpha\widetilde{\alpha}/\eta_k \\
          \beta\widetilde{\beta}/\eta_k & -\widetilde{\alpha}\beta/\eta_k \\
        \end{array}
      \right)+
      \left(
        \begin{array}{cc}
          -\alpha\widetilde{\beta}/\eta_k & \alpha\widetilde{\alpha}/\eta_k \\
          -\beta\widetilde{\beta}/\eta_k & \widetilde{\alpha}\beta/\eta_k\\
        \end{array}
      \right)}{z-z_k}+\mathcal{O}(1)\\
      &=&\mathcal{O}(1),\qquad\qquad(z\to z_k).
    \end{eqnarray*}
    The cases $z\to \overline{z}_k$ and $k\in\bigtriangleup$ are similar. (ii) and (iii) of \textbf{RHP[err]} are obvious.\\
    Now we turn to the analysis of \textbf{RHP[err]}. First of all we state two properties of the jump matrix $V^{(C)}$:
    \begin{equation}\label{equ: infty norm V-1}
        \|V^{(C)}-1\|_{L^{\infty}(\Sigma^{(1)})}\leq c \sqrt{t}e^{-8\sqrt{t}}
    \end{equation}
    \begin{equation}\label{equ: 2 norm V-1}
        \|V^{(C)}-1\|_{L^{2}(\Sigma^{(1)})}\leq c t^{1/4}e^{-8\sqrt{t}}
    \end{equation}
    These two estimates follow directly from the definition of $V^{(1)}\big|_{\Sigma^{(1)}}$ and $\|m^{(2)}(z)\|\leq G$ for $z\in\Sigma^{(1)}$ with a bound $G$, which does not depend on $x$ and $t$.
    It is a fact (see Chapter 7 in \cite{Ablowitz2003}), that the solution of \textbf{RHP[err]} is given by
    \begin{equation}\label{equ: solution m^err}
        C(z)=1+\frac{1}{2\pi\ii}\int_{\Sigma^{(1)}} \frac{\mu(\zeta)(V^{(C)}(\zeta)-1)}{\zeta-z}d\zeta,
    \end{equation}
    where $\mu\in L^2(\Sigma^{(1)})$ is the unique solution of \begin{equation}\label{equ: equ for mu}
        (1-C_V)\mu=1
    \end{equation}
    with $C_V:L^2(\Sigma^{(1)})\to L^2(\Sigma^{(1)})$ defined by
    \begin{equation*}
        (C_Vf)(x):=\lim_{\substack{z\to x\\z\in\ominus}}
        \frac{1}{2\pi\ii}\int_{\Sigma^{(1)}} \frac{f(\zeta)(V^{(C)}(\zeta)-1)}{\zeta-z}d\zeta.
    \end{equation*}
    By $z\in\ominus$ we indicate that the limit is to be taken non-tangentially from the minus (right) side of the (counter-clockwise) orientated contour $\Sigma^{(1)}$. In other words we set:
    \begin{equation*}
      \oplus := \bigcup_{\substack{z\in\mathcal{Z}\\z\notin \square(\xi)}}  B_{1/\sqrt{t}}(z),\qquad\quad
      \ominus := \Compl\setminus\overline{\oplus}.
    \end{equation*}
    We can write $C_V$ in terms of the Cauchy projection operator $C^-_{\Sigma^{(1)}}:L^2(\Sigma^{(1)})\to L^2(\Sigma^{(1)})$ which is defined by
    \begin{equation*}
        (C^-_{\Sigma^{(1)}}g)(x):=\lim_{\substack{z\to x\\z\in\ominus}}
        \frac{1}{2\pi\ii}\int_{\Sigma^{(1)}} \frac{g(\zeta)}{\zeta-z}d\zeta
    \end{equation*}
    and which has finite $L^2\to L^2$ operator norm. Moreover the operator norm is independent of $x$ and $t$. We have $C_Vf=C^-_{\Sigma^{(1)}}(f(V^{(C)}-1))$ and thus for any $f\in L^2(\Sigma^{(1)})$
    \begin{equation*}\label{equ: |C_V f|_L^2}
        \begin{aligned}
            \|C_V f\|_{L^2(\Sigma^{(1)})}&\leq\co
            \|f(V^{(C)}-1)\|_{L^2(\Sigma^{(1)})}\\
            &\leq\co \|V^{(C)}-1\|_{L^{\infty}(\Sigma^{(1)})}
            \|f\|_{L^2(\Sigma^{(1)})}.
        \end{aligned}
    \end{equation*}
    From (\ref{equ: infty norm V-1}) it follows that
    \begin{equation*}
        \|C_V\|_{L^2(\Sigma^{(1)})\to L^2(\Sigma^{(1)})}\leq c \sqrt{t}e^{-8\sqrt{t}}
    \end{equation*}
    with $c>0$ independent of $x$. We conclude that for $t$ sufficiently large $1-C_V$ is invertible and
    \begin{equation*}\label{equ: op norm of (1-C_V)^-1}
        \|(1-C_V)^{-1}\|_{L^2(\Sigma^{(1)})\to L^2(\Sigma^{(1)})}\leq\widetilde{c}
    \end{equation*}
    with $\widetilde{c}>0$ independent of $x$ and $t$. This implies for large $t$ that $\mu$ defined by equation (\ref{equ: equ for mu}) exists and satisfies
    \begin{equation}\label{equ: L^2 norm of mu}
        \|\mu\|_{L^2(\Sigma^{(1)})}\leq ct^{-1/4},
    \end{equation}
    where we have to take into account $\|1\|_{L^2(\Sigma^{(1)})}=4\pi w t^{-1/4}$ for some integer $0\leq w\leq N$.
    Equation (\ref{equ: solution m^err}) yields for large $z\in\Compl$
    \begin{eqnarray*}
      C(z)
       &=&  1-\frac{1}{2\pi\ii z}\int_{\Sigma^{(1)}} \frac{\mu(\zeta)(V^{(C)}(\zeta)-1)} {1-\frac{\zeta}{z}}d\zeta\\
       &=&  1-\frac{1}{2\pi\ii z}\int_{\Sigma^{(1)}} \mu(\zeta)(V^{(C)}(\zeta)-1) \left(1+\sum^{\infty}_{n=1} \left(\frac{\zeta}{z}\right)^n\right)d\zeta
    \end{eqnarray*}
    and thus we know how to choose the desired $C_1$ in (\ref{equ: m^1 approx m^sol}):
    \begin{equation*}
        C_1=-\frac{1}{2\pi\ii}\int_{\Sigma^{(1)}} \mu(\zeta)(V^{(C)}(\zeta)-1)d\zeta.
    \end{equation*}
    Making use of (\ref{equ: 2 norm V-1}), (\ref{equ: L^2 norm of mu}) and the H\"{o}lder inequality we conclude $\|C_1\|\leq c e^{-8\sqrt{t}}$.
\end{proof} 

\section{The Parabolic Cylinder RHP}\label{sec: cylinder}

The goal of our next modification $m^{(2)}\mapsto m^{(3)}$ is  the removal of the discontinuity on $\Real$. We will use the same technique presented for example in \cite{Cuccagna2014}, \cite{Dieng2008} and \cite{Jenkins2014}. The first step is the decomposition of the jump condition $V^{(2)}$ (see (\ref{equ: def V^2})). We write
\begin{equation*}
    V^{(2)}(z):=
        \left(
          \begin{array}{cc}
            1+|r^{(2)}(z)|^2 & e^{\overline{\phi}(z)}\overline{r^{(2)}}(z) \\
            e^{\phi(z)}r^{(2)}(z) & 1 \\
          \end{array}
        \right),\quad \text{with } r^{(2)}(z):=r(z) \prod_{k\in\bigtriangledown(\xi)} \left(\frac{z-z_k}{z-\overline{z}_k}\right)^2
\end{equation*}
and decompose now as follows:
\begin{equation}\label{equ: factor of v}
    V^{(2)}(z;x,t)=\left\{
                 \begin{array}{ll}
                   \widetilde{U}_L\widetilde{U}_0\widetilde{U}_R, & \hbox{for }z<\xi    \\
                   \widetilde{W}_L\widetilde{W}_R, & \hbox{for }z>\xi
                 \end{array}
               \right.,
\end{equation}
where
\begin{equation}\label{equ: ULURWLWR}
    \begin{aligned}
    &\widetilde{U}_L:=\left(
           \begin{array}{cc}
             1 & 0 \\
             e^{\phi(z;x,t)}\widetilde{R}_4(z)  & 1 \\
           \end{array}
         \right),
    \quad\widetilde{U}_0:=\left[1+|r(z)|^2\right]^{\sigma_3},
    \quad\widetilde{U}_R:=\left(
           \begin{array}{cc}
             1 & e^{-\phi(z;x,t)}\widetilde{R}_3(z) \\
             0 & 1 \\
           \end{array}
         \right), \\
    &\widetilde{W}_L:=\left(
           \begin{array}{cc}
             1 & e^{-\phi(z;x,t)} \widetilde{R}_6(z) \\
             0 & 1 \\
           \end{array}
         \right),
    \quad\widetilde{W}_R:=\left(
           \begin{array}{cc}
             1 & 0 \\
             e^{\phi(z;x,t)} \widetilde{R}_1(z) & 1 \\
           \end{array}
         \right).
    \end{aligned}
\end{equation}
and
\begin{equation}\label{equ: def tilde R_j}
    \widetilde{R}_4(z):=\frac{r^{(2)}(z)}{1+|r(z)|^2},\quad
    \widetilde{R}_3(z):=\frac{ \overline{r^{(2)}}(z)}{1+|r(z)|^2},\quad
    \widetilde{R}_6(z):=\overline{r^{(2)}}(z),\quad
    \widetilde{R}_1(z):=r^{(2)}(z).
\end{equation}
Note that $|r(z)|=|r^{(2)}(z)|$ $(z\in\Real)$ and moreover $c_1\|r\|_{H^s(\Real)}\leq\|r^{(2)}\|_{H^s(\Real)}\leq c_2\|r\|_{H^s(\Real)}$.\\
We will extend the functions $\widetilde{R}_j$ to special domains $\Omega_j$ which we define to be
\begin{eqnarray*}
  \Omega_1 &:=& \set{z\in\Compl|\arg(z-\xi) \in\left(0,\frac{\pi}{4}\right)} \\
  \Omega_2 &:=& \set{z\in\Compl|\arg(z-\xi) \in\left(\frac{\pi}{4},\frac{3\pi}{4}\right)} \\
  \Omega_3 &:=& \set{z\in\Compl|\arg(z-\xi) \in\left(\frac{3\pi}{4},\pi\right)} \\
  \Omega_4 &:=& \set{z\in\Compl|\arg(z-\xi) \in\left(\pi,\frac{5\pi}{4}\right)} \\
  \Omega_5 &:=& \set{z\in\Compl|\arg(z-\xi) \in\left(\frac{5\pi}{4},\frac{7\pi}{4}\right)} \\
  \Omega_6 &:=& \set{z\in\Compl|\arg(z-\xi) \in\left(\frac{7\pi}{4},2\pi\right)}
\end{eqnarray*}

Let now $R_j$ be the extensions of $\widetilde{R}_j$ into $\Omega_j$ (for $j\in\set{1,3,4,6}$).
\begin{equation}\label{equ: def Rj}
    \begin{aligned}
    R_1(z)&:=\cos(2\arg(z-\xi))\textbf{r}(z)+[1-\cos(2\arg(z-\xi))](z-\xi)^{-2\ii\nu_0}\widehat{r}_0\delta^2(z),\\
    R_3(z)&:=\cos(2\arg(z-\xi))\frac{\overline{\textbf{r}}(z)}{1+|\textbf{r}(z)|^2}\\&+ [1-\cos(2\arg(z-\xi))](z-\xi)^{2\ii\nu_0}\frac{\overline{\widehat{r}}_0}{1+|\widehat{r}_0|^2}\delta^{-2}(z),\\
    R_4(z)&:=\cos(2\arg(z-\xi))\frac{\textbf{r}(z)}{1+|\textbf{r}(z)|^2}\\&+ [1-\cos(2\arg(z-\xi))](z-\xi)^{-2\ii\nu_0}\frac{\widehat{r}_0}{1+|\widehat{r}_0|^2}\delta^2(z),\\
    R_6(z)&:=\cos(2\arg(z-\xi))\overline{\textbf{r}}(z)+[1-\cos(2\arg(z-\xi))](z-\xi)^{2\ii\nu_0}\overline{\widehat{r}}_0\delta^{-2}(z).
    \end{aligned}
\end{equation}
Here we set
\begin{equation}\label{equ: r bold}
    \textbf{r}(z):=\left\{
                     \begin{array}{ll}
                       r(\re z), & \hbox{if }\im z=0, \\
                       \varphi_{\im z}*r(\re z), & \hbox{if }\im z\neq 0,
                     \end{array}
                   \right.
\end{equation}
where $\varphi\in C^{\infty}_0(\Real,\Real)$ is of compact support and satisfies $\int \varphi dx=1$. We set $\varphi_{\eps}(x):=\eps^{-1}\varphi(\eps^{-1}x)$. $\varphi_{\eps}*r$ denotes the convolution of $\varphi$ and $r$. Further definitions are
\begin{equation}\label{equ: def }
    \begin{aligned}
    \nu_0(x,t) &:= -\frac{1}{2\pi}\log(1+|r(\xi)|^2)\\
      \widehat{r}_0(x,t) &:= r(\xi)e^{-2\ii\nu_0-2\beta_0}\\
        \delta(z;x,t)&:=\exp\left(\frac{1}{2\pi\ii} \int_{-\infty}^{\xi}\frac{\log(1+|r(y)|^2)} {y-z}dy\right)\\
  \beta_0(x,t) &:= \frac{1}{2\pi\ii} \int_{-\infty}^{\xi-1}\frac{\log(1+|r(y)|^2)} {y-\xi}dy\\&\qquad+\int_{\xi-1}^{\xi} \frac{\log(1+|r(y)|^2)-\log(1+|r(\xi)|^2)} {y-\xi}dy-\frac{\nu_0}{2\pi\ii}
    \end{aligned}
\end{equation}
By replacing $\widetilde{R}_j$ with $R_j$ in (\ref{equ: ULURWLWR}) we can obtain matrices $U_L,U_R,W_L$ and $W_R$, which are extensions into the same domains $\Omega_j$. Using these extensions we now define our third modification by:
\begin{equation}\label{equ: def m^3}
    m^{(3)}(z):=\left\{
       \begin{array}{ll}
         m^{(2)}(z)W_R(z)^{-1}\delta^{-\sigma_3}(z), & \hbox{ for }z\in\Omega_1, \\
         m^{(2)}(z)\delta^{-\sigma_3}(z), & \hbox{ for }z\in\Omega_2, \\
         m^{(2)}(z)U_R(z)^{-1}\delta^{-\sigma_3}(z), & \hbox{ for }z\in\Omega_3, \\
         m^{(2)}(z)U_L(z)\delta^{-\sigma_3}(z), & \hbox{ for }z\in\Omega_4, \\
         m^{(2)}(z)\delta^{-\sigma_3}(z), & \hbox{ for }z\in\Omega_5, \\
         m(z)^{(2)}W_L(z)\delta^{-\sigma_3}(z), &\hbox{ for }z\in\Omega_6.
       \end{array}
     \right.
\end{equation}
The price of this modification will be the loss of analyticity in $\Omega_1\cup\Omega_3\cup\Omega_4\cup\Omega_6$ and a jump on
\begin{equation}\label{equ: def Sigma3}
    \Sigma^{(3)}(x,t):=\bigcup_{n=1}^4\Sigma^{(3)}_n \cup\set{\xi}
\end{equation}
with $$\Sigma_1^{(3)}=e^{\ii\frac{\pi}{4}}\Real_++\xi,\quad \Sigma_2^{(3)}=e^{-\ii\frac{\pi}{4}}\Real_-+\xi,\quad \Sigma_3^{(3)}=e^{\ii\frac{\pi}{4}}\Real_-+\xi,\quad \Sigma_4^{(3)}=e^{-\ii\frac{\pi}{4}}\Real_++\xi,\quad$$ inheriting the orientation of $\Real_\pm$. In exchange for that the jump on $\Real$ is removed by (\ref{equ: def m^3}). In order to measure the non-analyticity of $m^{(3)}$ we use the operator $\overline{\partial}:=\frac{1}{2}(\partial_{\re z}+\ii\partial_{\im z})$:
\begin{lem}
    If $m^{(2)}(z)$ solves \textbf{RHP[2]}, then $m^{(3)}(z)$ defined in (\ref{equ: def m^3}) is a solution to the following $\overline{\partial}$-RHP:
    \begin{samepage}
\begin{framed}
\textbf{$\overline{\partial}$-RHP[3]:}\\Find for each $(x,t)\in\Real\times\Real$ a $2\times 2$-matrix valued function $\Compl\ni z\mapsto m^{(3)}(z;x,t)$ which satisfies
\begin{enumerate}[(i)]
  \item $m^{(3)}(z;x,t)$ is meromorphic in $\Omega_2\cup\Omega_5$ and continuous in $\Omega_1\cup\Omega_3\cup\Omega_4\cup\Omega_6\cup\Real$ (with respect to the parameter $z$) and $\overline{\partial}m^{(3)}=m^{(3)}W^{(3)}$, where
      \begin{equation}\label{equ: def W}
    W^{(3)}(z)=\left\{
            \begin{array}{ll}
              \left(
            \begin{array}{cc}
             0 & 0 \\
              -e^{\phi(z)}\delta^{-2}(z)\overline{\partial}R_1(z) & 0 \\
               \end{array}
               \right)
, & \hbox{ for } z\in \Omega_1,\\
              \left(
              \begin{array}{cc}
              0 & -e^{-\phi(z)}\delta^{2}(z)\overline{\partial}R_3(z) \\
              0 & 0 \\
              \end{array}
              \right)
, & \hbox{ for } z\in \Omega_3,\\
                \left(
                \begin{array}{cc}
                0 & 0 \\
                e^{\phi(z)} \delta^{-2}(z)\overline{\partial}R_4(z) & 0 \\
                \end{array}
                \right)
, & \hbox{ for } z\in \Omega_4,\\
                \left(
                \begin{array}{cc}
                0 & e^{-\phi(z)}\delta^{2}(z)\overline{\partial}R_6(z) \\
                0 & 0 \\
                \end{array}
                \right)
                , & \hbox{ for } z\in \Omega_6.
            \end{array}
            \right.
\end{equation}
  \item $m^{(3)}(z;x,t)=1+\mathcal{O}\left(\frac{1}{z}\right)$ as $|z|\to\infty$.
  \item If $\square(\xi)=\varnothing$, $m^{(2)}$ has no poles (i.e. $m^{(2)}$ is analytic in $\Omega_2\cup\Omega_5$). If $\square(\xi)$ consists of certain $z_k$ and $\overline{z}_{k}$ such that $k\in\bigtriangledown(\xi)$, $m^{(2)}$ has simple poles at these $z_k$ and $\overline{z}_{k}$ with:
      \begin{equation}\label{equ: Res sol j in nabla}
            \begin{aligned}
              \res_{z=z_k}m^{(2)}(z)&=\lim_{z\to z_k}m^{(2)}(z)
              \left(
                \begin{array}{cc}
                  0 & \frac{1}{c_k e^{\phi_k}T'(z_k)^2\delta(z_k)^{-2}} \\
                  0 & 0
                \end{array}
              \right),\\
              \res_{z=\overline{z}_k}m^{(2)}(z) &=\lim_{z\to \overline{z}_k}m^{(2)}(z)
              \left(
                \begin{array}{cc}
                  0 & 0 \\
                  \frac{-1}{\overline{c}_k e^{\overline{\phi}_k} \overline{T'(z_k)}^2\delta(\overline{z}_k)^{2}} & 0
                \end{array}
              \right).
            \end{aligned}
      \end{equation}
      If $\square(\xi)$ consists of certain $z_k$ and $\overline{z}_{k}$ such that $k\in\bigtriangleup(\xi)$, $m^{(2)}$ has simple poles at these $z_k$ and $\overline{z}_{k}$ with:
      \begin{equation}\label{equ: Res sol j in laplace}
            \begin{aligned}
              \res_{z=z_k}m^{(2)}(z)&=\lim_{z\to z_k}m^{(2)}(z)
              \left(
                \begin{array}{cc}
                  0 & 0 \\
                  c_k e^{\phi_k}T(z_k)^2\delta(z_k)^{-2}& 0
                \end{array}
              \right),\\
              \res_{z=\overline{z}_k}m^{(2)}(z) &=\lim_{z\to \overline{z}_k}m^{(2)}(z)
              \left(
                \begin{array}{cc}
                  0 & -\overline{c}_k e^{\overline{\phi}_k} \overline{T(z_k)}^2\delta(\overline{z}_k)^{2} \\
                  0 & 0
                \end{array}
              \right).
            \end{aligned}
      \end{equation}
  \item The non-tangential boundary values $m^{(3)}_{\pm}(z)$ exist for $z\in\Sigma^{(3)}$ and satisfy the jump relation $m^{(3)}_+=m^{(3)}_-V^{(3)}$, where
      \begin{equation}\label{equ: def V^3}
        V^{(3)}(z)=
        \left\{
             \begin{array}{ll}
               \delta^{\sigma_3}(z)W_R(z)\delta^{-\sigma_3}(z), & \hbox{ for }z\in\Sigma_1^{(3)}, \\
               \delta^{\sigma_3}(z)U_R(z)\delta^{-\sigma_3}(z), & \hbox{ for }z\in\Sigma_2^{(3)}, \\
               \delta^{\sigma_3}(z)U_L(z)\delta^{-\sigma_3}(z), & \hbox{ for }z\in\Sigma_3^{(3)}, \\
               \delta^{\sigma_3}(z)W_L(z)\delta^{-\sigma_3}(z), & \hbox{ for }z\in\Sigma_4^{(3)}.
             \end{array}
        \right.
      \end{equation}
\end{enumerate}
\end{framed}
\end{samepage}
\end{lem}
\begin{proof}
    For $z\in\Omega_1$ we have
    \begin{equation*}
            \overline{\partial}m^{(3)}= m^{(2)}\overline{\partial}W_R^{-1}\delta^{-\sigma_3}=
            m^{(3)}\delta^{\sigma_3}W_R \overline{\partial}W_R^{-1}\delta^{-\sigma_3} =m^{(3)}W^{(3)}.
    \end{equation*}
    The same calculation verifies (\ref{equ: def W}) for $z\in\Omega_3\cup\Omega_4\cup\Omega_6$. The analyticity of $\delta(z)$ implies that $m^{(3)}$ is meromorphic on $\Omega_2\cup\Omega_5$ if $m^{(2)}$ is meromorphic. Hence, in order to prove $(i)$ it remains to show that the jump of $m^{(2)}$ on $\Real$ is indeed removed (i.e. $m^{(3)}$ is continuous on $\Real$). For $z>\xi$ we have
    \begin{equation*}
        m^{(3)}_+(z)=m^{(2)}_+(z)\widetilde{W}_R^{-1}(z) \delta^{-\sigma_3}(z),\qquad m^{(3)}_-(z)=m^{(2)}_-(z)\widetilde{W}_L(z) \delta^{-\sigma_3}(z).
    \end{equation*}
    Taking into account that $\delta$ is analytic for $z>\xi$ and $m^{(2)}_+=m^{(2)}_-\widetilde{W}_L\widetilde{W}_R$ (see (\ref{equ: factor of v}) we find $m^{(3)}_+=m^{(3)}_-$. For $z<\xi$ the function $\delta$ has a jump and satisfies $\delta_+=\delta_-(1+|r|^2)$. This is a consequence of the Plemelj formulae (see \cite{Ablowitz2003}). Thus we have $\delta^{\sigma_3}_+=\delta^{\sigma_3}_-\widetilde{U}_0$ for $z<\xi$ and accordingly
    \begin{equation*}
        m^{(3)}_+=m^{(2)}_+\widetilde{U}_R^{-1} \delta_+^{-\sigma_3} =m^{(2)}_-[\widetilde{U}_L\widetilde{U}_0\widetilde{U}_R] \widetilde{U}_R^{-1} [\delta_-^{\sigma_3}\widetilde{U}_0]^{-1}
        =m^{(2)}_-\widetilde{U}_L\delta_-^{-\sigma_3}=m^{(3)}_-,
    \end{equation*}
    which completes the proof of ($i$). $(ii)$ follows from $\delta(z)=1+\mathcal{O}\left(\frac{1}{z}\right)$ as $|z|\to\infty$. $(iii)$ follows easily from the definition (\ref{equ: def m^3}) and the last point ($vi$) is also obvious.
\end{proof}
Our next goal is the elimination of the discontinuity of $m^{(3)}$ on $\Sigma^{(3)}$. The idea is very simple: We set
\begin{equation}\label{equ: def m^4}
    m^{(4)}(z):=m^{(3)}(z)[D(z)]^{-1},
\end{equation}
where $D$ is chosen such that it admits the same jump on $\Sigma^{(3)}$ as $m^{(3)}$ and leaves other properties of $m^{(3)}$ untouched. To be precise we take the solution of the following Riemann Hilbert problem:
\begin{samepage}
\begin{framed}
\textbf{RHP[$D$]:}\\Find for each $(x,t)\in\Real\times\Real$ a $2\times 2$-matrix valued function $\Compl\ni z\mapsto D(z;x,t)$ which satisfies
\begin{enumerate}[(i)]
  \item $D(z;x,t)$ is analytic in $\Compl\setminus\Sigma^{(3)}$ (with respect to the parameter $z$).
  \item $D(z;x,t)=1+\mathcal{O} \left(\frac{1}{z}\right)$ as $|z|\to\infty$.
  \item The non-tangential boundary values $D_{\pm}(z;x,t)$ exist for $z\in\Sigma^{(3)}$ and satisfy the jump relation $D_+=D_-V^{(3)}$.
\end{enumerate}
\end{framed}
\end{samepage}
As a consequence we have
\begin{equation*}
    m^{(4)}_+=m_+^{(3)}[D_+]^{-1}
    =m_-^{(3)}V^{(3)}[D_-V^{(3)}]^{-1}=
    m_-^{(3)}[D_-]^{-1}=m^{(4)}_-,\qquad
    z\in\Sigma^{(3)},
\end{equation*}
thus $m^{(4)}$ is indeed continuous on $\Sigma^{(3)}$. Furthermore the following lemma holds:
\begin{lem}
    If $m^{(3)}(z)$ solves \textbf{RHP[3]}, then $m^{(4)}(z)$ defined in (\ref{equ: def m^4}) is a solution to the following $\overline{\partial}$-RHP:
    \begin{samepage}
\begin{framed}
\textbf{$\overline{\partial}$-RHP[4]:}\\Find for each $(x,t)\in\Real\times\Real$ a $2\times 2$-matrix valued function $\Compl\ni z\mapsto m^{(4)}(z;x,t)$ which satisfies
\begin{enumerate}[(i)]
  \item $m^{(4)}(z;x,t)$ is meromorphic in $\Omega_2\cup\Omega_5$ and continuous in $\Omega_1\cup\Omega_3\cup \Omega_4\cup\Omega_6\cup\Real\cup\Sigma^{(3)}$ (with respect to the parameter $z$) and $\overline{\partial}m^{(4)}=m^{(4)}W^{(4)}$, where
      \begin{equation}\label{equ: def W^4}
    W^{(4)}(z)=D(z)W^{(3)}(z) [D(z)]^{-1}
\end{equation}
  \item $m^{(4)}(z;x,t)=1+\mathcal{O}\left(\frac{1}{z}\right)$ as $|z|\to\infty$.
  \item If $\square(\xi)=\varnothing$, $m^{(4)}$ has no poles (i.e. $m^{(4   )}$ is analytic in $\Omega_2\cup\Omega_5$). If $\square(\xi)$ consists of certain $z_k$ and $\overline{z}_{k}$ such that $k\in\bigtriangledown(\xi)$, $m^{(4)}$ has simple poles at these $z_k$ and $\overline{z}_{k}$ with:
      \begin{equation}\label{equ: Res sol j in nabla}
            \begin{aligned}
              \res_{z=z_k}m^{(4)}(z)&=\lim_{z\to z_k}m^{(4)}(z)
              D(z_k)\left(
                \begin{array}{cc}
                  0 & \frac{1}{c_k e^{\phi_k}T'(z_k)^2\delta(z_k)^{-2}} \\
                  0 & 0
                \end{array}
              \right)[D(z_k)]^{-1},\\
              \res_{z=\overline{z}_k}m^{(4)}(z) &=\lim_{z\to \overline{z}_k}m^{(4)}(z)
              D(\overline{z}_k)\left(
                \begin{array}{cc}
                  0 & 0 \\
                  \frac{-1}{\overline{c}_k e^{\overline{\phi}_k} \overline{T'(z_k)}^2\delta(\overline{z}_k)^{2}} & 0
                \end{array}
              \right)[D(\overline{z}_k)]^{-1}.
            \end{aligned}
      \end{equation}
      If $\square(\xi)$ consists of certain $z_k$ and $\overline{z}_{k}$ such that $k\in\bigtriangleup(\xi)$, $m^{(4)}$ has simple poles at these $z_k$ and $\overline{z}_{k}$ with:
      \begin{equation}\label{equ: Res sol j in laplace}
            \begin{aligned}
              \res_{z=z_k}m^{(4)}(z)&=\lim_{z\to z_k}m^{(4)}(z)
              D(z_k)\left(
                \begin{array}{cc}
                  0 & 0 \\
                  c_k e^{\phi_k}T(z_k)^2\delta(z_k)^{-2}& 0
                \end{array}
              \right)[D(z_k)]^{-1},\\
              \res_{z=\overline{z}_k}m^{(4)}(z) &=\lim_{z\to \overline{z}_k}m^{(4)}(z)
              D(\overline{z}_k)\left(
                \begin{array}{cc}
                  0 & -\overline{c}_k e^{\overline{\phi}_k} \overline{T(z_k)}^2\delta(\overline{z}_k)^{2} \\
                  0 & 0
                \end{array}
              \right)[D(\overline{z}_k)]^{-1}.
            \end{aligned}
      \end{equation}
\end{enumerate}
\end{framed}
\end{samepage}
\end{lem}
The proof of this lemma is elementary and we will skip it here. Instead we have to say a word on  \textbf{RHP[$D$]}. It can be solved explicitly and the solution has been worked out for example in \cite{Cuccagna2014}, \cite{Deift1994}, \cite{Dieng2008} or \cite{Jenkins2011}.
\begin{lem}\label{lem: RHP P}
    \begin{enumerate}
      \item \textbf{RHP[$D$]} has an unique solution,
      \item $\|D(\cdot;x,t)\|_{L^{\infty}(\Compl)}\leq C$ ($C$ does not depend on $x$ and $t$),
      \item $D(z;x,t)=1+\frac{D_1(x,t)}{z}+\mathcal{O}(z^{-2})$ as $|z|\to\infty$ and $|D_1(x,t)|\leq c \eps t^{-1/2}$.
    \end{enumerate}
\end{lem}
Using the transformation $\zeta\leftrightarrow\sqrt{8t}(z-z_0)$ we can transform \textbf{RHP[$D$]} into the \textbf{Parabolic Cylinder RHP}:
\begin{equation}\label{equ: RHP for P}
    \left\{
      \begin{array}{ll}
        P(\zeta)\text{ is analytic for }\arg(\zeta)\notin \set{\pi/4,3\pi/4,5\pi/4,7\pi/4},\\
        P_+(\zeta)=P_-(\zeta)V_{P}(\zeta)\text{ for }\arg(\zeta)\in \set{\pi/4,3\pi/4,5\pi/4,7\pi/4},\\
        P(\zeta)\to 1\text{ as }\zeta\to\infty,
      \end{array}
    \right.
\end{equation}
where
\begin{equation}\label{equ: def V P}
    V_P(\zeta):=\left\{
                  \begin{array}{ll}
                    \left(
                       \begin{array}{cc}
                         1 & 0 \\
                         r_0\zeta^{-2\ii\nu_0}e^{\ii\zeta^2/2} & 1 \\
                       \end{array}
                     \right)
, & \hbox{for }\arg(\zeta)=\pi/4, \\
                    \left(
                      \begin{array}{cc}
                        1 & \frac{\overline{r}_0}{1+|r_0|^2}\zeta^{2\ii\nu_0}e^{-\ii\zeta^2/2} \\
                        0 & 1 \\
                      \end{array}
                    \right)
, & \hbox{for }\arg(\zeta)=3\pi/4, \\
                    \left(
                      \begin{array}{cc}
                        1 &  0 \\
                        \frac{r_0}{1+|r_0|^2}\zeta^{-2\ii\nu_0}e^{\ii\zeta^2/2} & 1 \\
                      \end{array}
                    \right)
, & \hbox{for }\arg(\zeta)=5\pi/4, \\
                    \left(
                      \begin{array}{cc}
                        1 & \overline{r}_0\zeta^{2\ii\nu_0}e^{-\ii\zeta^2/2} \\
                        0 & 1 \\
                      \end{array}
                    \right)
, & \hbox{for }\arg(\zeta)=7\pi/4.
                  \end{array}
                \right.,
\end{equation}
The statements of Lemma \ref{lem: RHP P} on $D(z)=P(\sqrt{8t}(z-z_0))$ are consequences of analogous statements on $P$ which are well known and derived in the references mentioned above.

\section{The $\overline{\partial}$-method}\label{sec: dbar}
In this section we show that for large $t$ we can forget about the $\overline{\partial}$ part. The proof is taken from \cite{Dieng2008}. We consider \textbf{$\overline{\partial}$-RHP[4]} with $W^{(4)}\equiv 0$:
\begin{samepage}
\begin{framed}
\textbf{RHP[5]:}\\Find for each $(x,t)\in\Real\times\Real$ a $2\times 2$-matrix valued function $\Compl\ni z\mapsto m^{(4)}(z;x,t)$ which satisfies
\begin{enumerate}[(i)]
  \item $m^{(5)}(z;x,t)$ is meromorphic in $\Compl$.
  \item $m^{(5)}(z;x,t)=1+\mathcal{O}\left(\frac{1}{z}\right)$ as $|z|\to\infty$.
  \item If $\square(\xi)=\varnothing$, $m^{(5)}$ has no poles (i.e. $m^{(5)}$ is analytic in $\Omega_2\cup\Omega_5$). If $\square(\xi)$ consists of certain $z_k$ and $\overline{z}_{k}$ such that $k\in\bigtriangledown(\xi)$, $m^{(5)}$ has simple poles at these $z_k$ and $\overline{z}_{k}$ with:
      \begin{equation}\label{equ: Res sol j in nabla}
            \begin{aligned}
              \res_{z=z_k}m^{(5)}(z)&=\lim_{z\to z_k}m^{(5)}(z)
              D(z_k)\left(
                \begin{array}{cc}
                  0 & \frac{1}{c_k e^{\phi_k}T'(z_k)^2\delta(z_k)^{-2}} \\
                  0 & 0
                \end{array}
              \right)[D(z_k)]^{-1},\\
              \res_{z=\overline{z}_k}m^{(5)}(z) &=\lim_{z\to \overline{z}_k}m^{(5)}(z)
              D(\overline{z}_k)\left(
                \begin{array}{cc}
                  0 & 0 \\
                  \frac{-1}{\overline{c}_k e^{\overline{\phi}_k} \overline{T'(z_k)}^2\delta(\overline{z}_k)^{2}} & 0
                \end{array}
              \right)[D(\overline{z}_k)]^{-1}.
            \end{aligned}
      \end{equation}
      If $\square(\xi)$ consists of certain $z_k$ and $\overline{z}_{k}$ such that $k\in\bigtriangleup(\xi)$, $m^{(5)}$ has simple poles at these $z_k$ and $\overline{z}_{k}$ with:
      \begin{equation}\label{equ: Res sol j in laplace}
            \begin{aligned}
              \res_{z=z_k}m^{(5)}(z)&=\lim_{z\to z_k}m^{(5)}(z)
              D(z_k)\left(
                \begin{array}{cc}
                  0 & 0 \\
                  c_k e^{\phi_k}T(z_k)^2\delta(z_k)^{-2}& 0
                \end{array}
              \right)[D(z_k)]^{-1},\\
              \res_{z=\overline{z}_k}m^{(5)}(z) &=\lim_{z\to \overline{z}_k}m^{(5)}(z)
              D(\overline{z}_k)\left(
                \begin{array}{cc}
                  0 & -\overline{c}_k e^{\overline{\phi}_k} \overline{T(z_k)}^2\delta(\overline{z}_k)^{2} \\
                  0 & 0
                \end{array}
              \right)[D(\overline{z}_k)]^{-1}.
            \end{aligned}
      \end{equation}
\end{enumerate}
\end{framed}
\end{samepage}
\begin{lem}\label{lem: m^4 approx m^5}
    Let $m^{(4)}$ solve \textbf{$\overline{\partial}$-RHP[4]} and $m^{(5)}$ be a solution to \textbf{RHP[5]}. Then there is a matrix $E_1(x,t)$ for which
    \begin{equation*}
        \|E_1\|\leq c t^{-1/2}\qquad(t>0)
    \end{equation*}
    (with $c>0$ independent of $x$) holds and such that
    \begin{equation}\label{equ: m^1 approx m^sol}
         m^{(4)}(z)= \left[ 1+\frac{E_1}{z}+\mathcal{O}\left(\frac{1}{z^2}\right) \right]m^{(5)}(z)
    \end{equation}
    as $|z|\to\infty$.
\end{lem}
\begin{proof}
    It can be easily verified that $E(z):=m^{(4)}(z)[m^{(5)}(z)]^{-1}$ solves the following $\overline{\partial}$-problem:
    \begin{samepage}
    \begin{framed}
        \textbf{$\overline{\partial}$-problem for $E$:}
        \\Find for each $(x,t)\in\Real\times\Real$ a $2\times 2$-matrix valued function $\Compl\ni z\mapsto E(z;x,t)$ which satisfies
        \begin{enumerate}[(i)]
           \item $E(z)$ is continuous in $\Compl$,
           \item $E(z;x,t)=1+\frac{E_1}{z} \mathcal{O}\left(\frac{1}{z^2}\right)$ as $|z|\to\infty$, $z\in\Omega_2\cup\Omega_5$,
           \item $\overline{\partial} E=EW$ with $W=m^{(5)}W^{(4)}[m^{(5)}]^{-1}$.
    \end{enumerate}
    \end{framed}
    \end{samepage}
    As described in \cite[Section~3]{Pelinovsky2014}, the solution $E$ is obtained by taking the unique solution of $E=1+J(E)$. The operator $J:L^{\infty}(\Compl)\to L^{\infty}(\Compl)\cap C^0(\Compl)$ is defined by
    \begin{equation}\label{equ: def J}
        JH(z):=\frac{1}{\pi}\int_{\Compl} \frac{H(\varsigma)W(\varsigma)}{\varsigma-z} dA(\varsigma).
    \end{equation}
    Using estimates on $\overline{\partial}R_j$ (see Proposition 3.6 in \cite{Cuccagna2014} it can be proved that $\|J\|_{L^{\infty}(\Compl)\to L^{\infty}(\Compl)}\leq c t^{\frac{1-2s}{4}} $ ($c$ independent of $x$). Hence, $\|E\|_{L^{\infty}(\Compl)}= \|(1-J)^{-1}1\|_{L^{\infty}(\Compl)}$ is bounded uniformly in $(x,t)$ for sufficiently large $t$. As a consequence we find
    \begin{equation*}
        |E_1|=\left|\frac{1}{\pi}\int_{\Compl}EWdA\right|\leq c\sum_{j\in\set{1,3,4,6}}\int_{\Omega_j}|W|dA\leq c t^{-\frac{1+2s}{4}}.
    \end{equation*}
    For the latter inequality see the calculations in the proof of Lemma 3.9 in \cite{Cuccagna2014}.
\end{proof} 

\section{The last step}\label{sec: last step}
Note that \textbf{RHP[5]} does not describe a soliton or breather because the residuum conditions are not those of solitons. However, we have $P(z_k)\to1$ and $\delta(z_k)\to 1/\Lambda^+_k$ for $|\xi-\re(z_k)|<1/\sqrt{t}$ and $t\to\infty$.
For a small $\rho>0$ such that $\bigcap_{z\in\mathcal{Z}}B_{\rho}(z)=\varnothing$ we set:
\begin{equation}\label{equ: def m^6}
    m^{(6)}(z):=
    \left\{
      \begin{array}{ll}
        m^{(5)}(z)D(z_k)\left( \delta(z_k)\Lambda_k^+\right)^{\sigma_3}, & \hbox{if }z\in B_{\rho}(z_k), z_k\in\square(\xi), \\
        m^{(5)}(z)D(\overline{z}_k)\left( \delta(\overline{z}_k)/ \overline{\Lambda}^+_k\right)^{\sigma_3}, & \hbox{if }z\in B_{\rho}(\overline{z}_k), \overline{z}_k\in\square(\xi), \\
        m^{(5)}(z), & \hbox{else.}
      \end{array}
    \right.
\end{equation}
Note that  $m^{(6)}$ differs from $m^{(5)}$ only if $\square(\xi)\neq\varnothing$. For $\square(\xi)\neq\varnothing$ a discontinuity  appears on
\begin{equation*}
    \Sigma^{(6)}(\xi)=\bigcup_{z\in\square(\xi)}\partial B_{\rho}(z).
\end{equation*}

\begin{lem}
    If $m^{(5)}(z)$ solves \textbf{RHP[5]}, then $m^{(6)}(z)$ defined in (\ref{equ: def m^6}) is a solution to the following RHP:
\begin{samepage}
\begin{framed}
\textbf{RHP[6]:}\\Find for each $(x,t)\in\Real\times\Real$ a $2\times 2$-matrix valued function $\Compl\ni z\mapsto m^{(6)}(z;x,t)$ which satisfies
\begin{enumerate}[(i)]
  \item $m^{(6)}(z;x,t)$ is meromorphic in $\Compl\setminus\Sigma^{(6)}$.
  \item $m^{(6)}(z;x,t)=1+\mathcal{O}\left(\frac{1}{z}\right)$ as $|z|\to\infty$.
  \item If $\square(\xi)=\varnothing$, $m^{(6)}$ has no poles. If $\square(\xi)$ consists of certain $z_k$ and $\overline{z}_{k}$ such that $k\in\bigtriangledown(\xi)$, $m^{(5)}$ has simple poles at these $z_k$ and $\overline{z}_{k}$ with:
      \begin{equation}\label{equ: Res sol j in nabla}
            \begin{aligned}
              \res_{z=z_k}m^{(6)}(z)&=\lim_{z\to z_k}m^{(6)}(z)
              \left(
                \begin{array}{cc}
                  0 & \frac{1}{c_k (\Lambda_k^+)^2  e^{\phi_k}T'(z_k)^2} \\
                  0 & 0
                \end{array}
              \right),\\
              \res_{z=\overline{z}_k}m^{(6)}(z) &=\lim_{z\to \overline{z}_k}m^{(6)}(z)
              \left(
                \begin{array}{cc}
                  0 & 0 \\
                  \frac{-1}{\overline{c}_k (\overline{\Lambda}_k^+)^2 e^{\overline{\phi}_k} \overline{T'(z_k)}^2} & 0
                \end{array}
              \right).
            \end{aligned}
      \end{equation}
      If $\square(\xi)$ consists of certain $z_k$ and $\overline{z}_{k}$ such that $k\in\bigtriangleup(\xi)$, $m^{(6)}$ has simple poles at these $z_k$ and $\overline{z}_{k}$ with:
      \begin{equation}\label{equ: Res sol j in laplace}
            \begin{aligned}
              \res_{z=z_k}m^{(6)}(z)&=\lim_{z\to z_k}m^{(6)}(z)
              \left(
                \begin{array}{cc}
                  0 & 0 \\
                  c_k (\Lambda_k^+)^2 e^{\phi_k}T(z_k)^2& 0
                \end{array}
              \right),\\
              \res_{z=\overline{z}_k}m^{(6)}(z) &=\lim_{z\to \overline{z}_k}m^{(6)}(z)
              \left(
                \begin{array}{cc}
                  0 & -\overline{c}_k (\overline{\Lambda}_k^+)^2 e^{\overline{\phi}_k} \overline{T(z_k)}^2 \\
                  0 & 0
                \end{array}
              \right).
            \end{aligned}
      \end{equation}
  \item The non-tangential boundary values $m_{\pm}^{(6)}(z)$ exist for $z\in\Sigma^{(6)}$ and satisfy the jump relation $m_+^{(6)}=m_-^{(6)}V^{(6)}$, where
      \begin{equation}\label{equ: def V^6}
         V^{(6)}(z)=
         \left\{
          \begin{array}{ll}
            D(z_k)\left( \delta(z_k)\Lambda_k^+\right)^{\sigma_3}, & \hbox{if }z\in \partial B_{\rho}(z_k), z_k\in\square(\xi), \\
            D(\overline{z}_k)\left( \delta(\overline{z}_k)/ \overline{\Lambda}^+_k\right)^{\sigma_3}, & \hbox{if }z\in B_{\rho}(\overline{z}_k),\overline{z}_k\in\square(\xi).
          \end{array}
         \right.
      \end{equation}
\end{enumerate}
\end{framed}
\end{samepage}
\end{lem}
The proof is elementary.\\ We are now arrived at our last step. Later in Lemma \ref{lem: m^6 approx m^7} we will show that we may replace $V^{(6)}$ in (\ref{equ: def V^6}) by $1$ which is a consequence of $P(z_k)\to 1$ and $\delta(z_k)\to 1/\Lambda^+_k$ for $|\xi-\re(z_k)|<1/\sqrt{t}$ and $t\to\infty$. Since the condition $|\xi-\re(z_k)|<1/\sqrt{t}$ is fulfilled whenever $\square(\xi)\neq\varnothing$ we thus have $ V^{(6)}\to 1$:
\begin{prop}\label{prop: V^6 to 1}
    There exist constants $c,T>0$ such that
    \begin{equation}\label{equ: V^6 to 1}
            \|V^{(6)}-1\|_{L^{\infty}(\Sigma^{(6)})}\leq ct^{-1/2},\qquad
            \|V^{(6)}-1\|_{L^{2}(\Sigma^{(6)})}\leq ct^{-1/2},
    \end{equation}
    for $t>T$.
\end{prop}
\begin{proof}
    Obviously the $L^2$-estimate of (\ref{equ: V^6 to 1}) follows from the $L^{\infty}$-estimate, due to $\meas(\Sigma^{(6)})<\infty$. Furthermore the proposition is trivial in the case of $\square(\xi)=\varnothing$, where we have $\Sigma^{(6)}=\varnothing$. Let us now assume that $z_k\in\square(\xi)\cap\mathcal{Z}_+$ and thus $|\xi-\re(z_k)|<1/\sqrt{t}$:
    \begin{eqnarray*}
      \left|\delta(z_k)\Lambda_k^+-1\right| &=& \left|\exp\left(\frac{1}{2\pi\ii} \int_{\re(z_{k})}^{\xi} \frac{\log(1+|r(\varsigma)|^2)} {\varsigma-z_{k}}d\varsigma\right)-1\right| \\
       &\leq& c\left|\int_{\re(z_{k})}^{\xi} \frac{\log(1+|r(\varsigma)|^2)} {\varsigma-z_{k}}d\varsigma\right| \\
       &\leq & c\frac{|\xi-\re(z_k)|}{\im(z_k)} \log(1+\|r\|_{H^s(\Real)}^2)\leq ct^{-1/2}
    \end{eqnarray*}
    Analogously we have $\left|\delta(\overline{z}_k)/ \overline{\Lambda}^+_k-1\right|\leq ct^{-1/2}$. Additionally we take $|D(z_k)-1|\leq ct^{-1/2}$ and $|D(\overline{z}_k)-1|\leq ct^{-1/2}$ from Lemma \ref{lem: RHP P} and thus the proof is completed.
\end{proof}
If we omit the jump on $\Sigma^{(6)}$ in \textbf{RHP[6]} we get:
\begin{samepage}
\begin{framed}
\textbf{RHP[7]:}\\Find for each $(x,t)\in\Real\times\Real$ a $2\times 2$-matrix valued function $\Compl\ni z\mapsto m^{(7)}(z;x,t)$ which satisfies
\begin{enumerate}[(i)]
  \item $m^{(7)}(z;x,t)$ is meromorphic in $\Compl$.
  \item $m^{(7)}(z;x,t)=1+\mathcal{O}\left(\frac{1}{z}\right)$ as $|z|\to\infty$.
  \item If $\square(\xi)=\varnothing$, $m^{(7)}$ has no poles (i.e. $m^{(7)}$ is entire). If $\square(\xi)$ consists of certain $z_k$ and $\overline{z}_{k}$ such that $k\in\bigtriangledown(\xi)$, $m^{(5)}$ has simple poles at these $z_k$ and $\overline{z}_{k}$ with:
      \begin{equation}\label{equ: Res sol j in nabla}
            \begin{aligned}
              \res_{z=z_k}m^{(7)}(z)&=\lim_{z\to z_k}m^{(7)}(z)
              \left(
                \begin{array}{cc}
                  0 & \frac{1}{c_k (\Lambda_k^+)^2  e^{\phi_k}T'(z_k)^2} \\
                  0 & 0
                \end{array}
              \right),\\
              \res_{z=\overline{z}_k}m^{(7)}(z) &=\lim_{z\to \overline{z}_k}m^{(7)}(z)
              \left(
                \begin{array}{cc}
                  0 & 0 \\
                  \frac{-1}{\overline{c}_k (\overline{\Lambda}_k^+)^2 e^{\overline{\phi}_k} \overline{T'(z_k)}^2} & 0
                \end{array}
              \right).
            \end{aligned}
      \end{equation}
      If $\square(\xi)$ consists of certain $z_k$ and $\overline{z}_{k}$ such that $k\in\bigtriangleup(\xi)$, $m^{(7)}$ has simple poles at these $z_k$ and $\overline{z}_{k}$ with:
      \begin{equation}\label{equ: Res sol j in laplace}
            \begin{aligned}
              \res_{z=z_k}m^{(7)}(z)&=\lim_{z\to z_k}m^{(7)}(z)
              \left(
                \begin{array}{cc}
                  0 & 0 \\
                  c_k (\Lambda_k^+)^2 e^{\phi_k}T(z_k)^2& 0
                \end{array}
              \right),\\
              \res_{z=\overline{z}_k}m^{(7)}(z) &=\lim_{z\to \overline{z}_k}m^{(7)}(z)
              \left(
                \begin{array}{cc}
                  0 & -\overline{c}_k (\overline{\Lambda}_k^+)^2 e^{\overline{\phi}_k} \overline{T(z_k)}^2 \\
                  0 & 0
                \end{array}
              \right).
            \end{aligned}
      \end{equation}
\end{enumerate}
\end{framed}
\end{samepage}
The following Lemma is comparable to Lemma \ref{lem: m^1 approx m^sol}:
\begin{lem}\label{lem: m^6 approx m^7}
    Let $m^{(6)}$ solve \textbf{RHP[6]} and $m^{(7)}$ be a solution to \textbf{RHP[7]}. Then there is a matrix $F_1(x,t)$ for which
    \begin{equation*}
        \|F_1\|\leq c t^{-1/2}\qquad(t>0)
    \end{equation*}
    (with $c>0$ independent of $x$) holds and such that
    \begin{equation}\label{equ: m^6 approx m^7}
         m^{(6)}(z)= \left[ 1+\frac{F_1}{z}+\mathcal{O}\left(\frac{1}{z^2}\right) \right]m^{(7)}(z)
    \end{equation}
    as $|z|\to\infty$
\end{lem}
\begin{proof}
    We set $F(z):=m^{(6)}(z)[m^{(7)}(z)]^{-1}$ which admits a solution of the following Riemann Hilbert problem:
    \begin{samepage}
    \begin{framed}
        \textbf{RHP[F]}
        \begin{enumerate}[(i)]
          \item $F$ is analytic  in $\Compl\setminus\Sigma^{(1)}$.
          \item $F(z)= 1+ \mathcal{O}\left(\frac{1}{z}\right)$ as $|z|\to\infty$.
          \item The non-tangential boundary values $F_{\pm}(z)$ exist for $z\in\Sigma^{(6)}$ and satisfy the jump relation $F_+=F_-V^{(F)}$, where
              \begin{equation*}
                V^{(F)}(z)=m^{(7)}(z)V^{(6)}(z) \left[m^{(7)}(z)\right]^{-1}.
              \end{equation*}
        \end{enumerate}
    \end{framed}
    \end{samepage}
    Now we proceed as in the proof of Lemma \ref{lem: m^1 approx m^sol}. That is firstly to find $\eta\in L^2(\Sigma^{(6)})$ such that
    \begin{equation*}\label{equ: solution eta}
        \eta(x)=1+ \lim_{\substack{z\to x\\z\in\ominus}}\frac{1}{2\pi\ii}\int_{\Sigma^{(6)}} \frac{\eta(\zeta)(V^{(F)}(\zeta)-1)}{\zeta-z}d\zeta, \qquad z\in\Sigma^{(6)}.
    \end{equation*}
    The next step is to observe that
    \begin{equation*}
        F_1=-\frac{1}{2\pi\ii}\int_{\Sigma^{(6)}} \eta(\zeta)(V^{(F)}(\zeta)-1)d\zeta.
    \end{equation*}
    Proposition \ref{prop: V^6 to 1} ensures the existence of $\eta$ and the required estimate $\|F_1\|\leq c t^{-1/2}$. Note that (\ref{equ: V^6 to 1}) is also true for $V^{(F)}$ instead of $V^{(6)}$ since $\|m^{(7)}(\,\cdot\,;x,t)\|_{L^{\infty}(\Sigma^{(6)})}\leq C$ with $C$ independent of $x$ and $t$.
\end{proof} 

\section{Proof of Theorem \ref{thm: mainthm}}\label{sec: proof}
In order to prove Theorem \ref{thm: mainthm} we firstly assume $u_0\in H^1(\Real)\cap L^{2,s}(\Real)$. Furthermore we assume (\ref{equ: u minus sol}), which gives us $u_0\in\mathcal{G}_N$ (see Theorem \ref{thm: scattering} (i)). Then the scattering data $(r;z'_1,.., z'_N;c'_1,..,c'_N)$ of $u_0$ can be calculated and the solution $u(x,t)$ can be obtained by applying the reconstruction formula (\ref{equ: rec}) to the solution $m$ of \textbf{RHP[NLS]}. Starting from this $m$ we consider our chain of manipulations $m\to m^{(1)}\to\ldots\to m^{(7)}$ and in each step we calculate the associated potential
$$u^{(j)}(x,t):=2\ii\lim_{|z|\to\infty} z[m^{(j)}(z;x,t)]_{12}\qquad j=1,...,7.$$
Applying successively (\ref{equ: def m^1}), Lemma \ref{lem: m^1 approx m^sol}, (\ref{equ: def m^3}), (\ref{equ: def m^4}), Lemma \ref{lem: m^4 approx m^5}, (\ref{equ: def m^6}) and finally Lemma \ref{lem: m^6 approx m^7}, we arrive at
\begin{equation}\label{equ: expansion of u}
    \begin{aligned}
      u(x,t)&=u^{(1)}(x,t)\\
            &=u^{(2)}(x,t)+2\ii[C_1(x,t)]_{12}\\ &=u^{(3)}(x,t)+2\ii[C_1(x,t)]_{12}\\ &=u^{(4)}(x,t)+2\ii[C_1(x,t)+D_1(x,t)]_{12}\\ &=u^{(5)}(x,t)+2\ii[C_1(x,t)+D_1(x,t)+E_1(x,t)]_{12}\\ &=u^{(6)}(x,t)+2\ii[C_1(x,t)+D_1(x,t)+E_1(x,t)]_{12}\\ &=u^{(7)}(x,t) +2\ii[C_1(x,t)+D_1(x,t)+E_1(x,t)+F_1(x,t)]_{12}
    \end{aligned}
\end{equation}
The estimates of Lemmata \ref{lem: m^1 approx m^sol},\ref{lem: RHP P}, \ref{lem: m^4 approx m^5} and \ref{lem: m^6 approx m^7} yield
\begin{equation*}
        \left\| u(\cdot,t)-u^{(7)}(\cdot,t)
            \right\|_{L^{\infty}(\Real)}<C\epsilon |t|^{-\frac12}.
\end{equation*}
Now the remaining question is wether $u^{(7)}$ approximates a $N$-soliton $u^{(sol)}_+$ and we have to specify its parameters. We claim that the poles of the approximating soliton $u^{(sol)}_+$ are the same of $u_0$ and the coupling constants are given by $c_j^{+}=c_j'(\Lambda_j^+)^2$ (where $c'_j$ are the coupling constants of $u_0$). The proof is easy if we use again the above manipulations. Therefore we consider the solution $\widetilde{m}$ of \textbf{RHP[NLS]} with parameters $(0;z'_1,.., z'_N;c^+_1,..,c^+_N)$ such that $u^{(sol)}_+=2\ii \lim z [\widetilde{m}]_{12}$. Starting from $\widetilde{m}$ our manipulations $\widetilde{m}\to...\to\widetilde{m}^{(7)}$ then yield
\begin{equation*}
    u^{(sol)}_+=\widetilde{u}^{(7)}+2\ii[\widetilde{C}_1]_{12}
\end{equation*}
and moreover $\widetilde{u}^{(7)}=u^{(7)}$. Thus (\ref{equ: stability of n solitons}) follows. $|z_j-z'_j|<C\epsilon$ and $|c_j-c'_j|<C\epsilon$ are consequences of the Lipschitz continuity of the scattering transformation. $|c_j-c^+_j|<C\epsilon$ follows if we also use $$|1-\Lambda_j^+|\leq c\left|\int_{-\infty}^{\re(z_{j})} \frac{\log(1+|r(\varsigma))|^2} {\varsigma-z'_{j}}d\varsigma\right|\leq C \|r\|_{L^2(\Real)}^2\leq C\epsilon.$$
Thus the proof of our main result is done for $u_0\in H^1(\Real)$ and $t\to +\infty$. Density arguments like those in \cite{Cuccagna2014} prove the statement for $u_0\in H^1(\Real)\cap L^{2,s}(\Real)$ but $u_0\notin H^1(\Real)$.\\
The case $t\to -\infty$ can be handled as follows. If $u(x,t)$ solves (\ref{equ: nls}) then $\check{u}(x,t):=\overline{u}(x,-t)$ is also a solution to the NLS equation with $\check{u}(x,0)=\overline{u}_0$. Assuming that $(r(z);z'_1,...,z'_N;c'_1,...,c'_N)$ are the scattering data of $u_0$, we know due to the symmetry of (\ref{equ: vx=Pv}) that $\overline{u}_0$ admits scattering data $(\check{r};\check{z}_1,...,\check{z}_N; \check{c}_1,...,\check{c}_N)$ with
\begin{equation*}
    \check{r}(z)=\overline{r}(-z), \quad\check{z}_j=-\overline{z}_j', \quad\check{c}_j=-\overline{c}_j'.
\end{equation*}
By the above calculations we know that $$\|\check{u}(\cdot,t)-\check{u}^{(sol)}_+(\cdot,t) \|_{L^{\infty}(\Real)}< C \epsilon t^{-1/2}\quad\text{as }t\to+\infty,$$
where $\check{u}^{(sol)}_+$ is the soliton associated to the scattering data $(0;\check{z}_1,...,\check{z}_N; \check{c}^+_1,...,\check{c}^+_N)$ with
\begin{equation}\label{equ: check c^+_j}
    \check{c}^+_j=\check{c}_j(\check{\Lambda}_j^+)^2,\quad
    \check{\Lambda}_{j}^{+}:=\exp\left(-\frac{1}{2\pi\ii} \int_{-\infty}^{\re(\check{z}_{j})} \frac{\log(1+|\check{r}(\varsigma))|^2} {\varsigma-\check{z}_{j}}d\varsigma
          \right).
\end{equation}
After inverse transformation we arrive at
$$\|u(\cdot,t)-u^{(sol)}_-(\cdot,t) \|_{L^{\infty}(\Real)}< C \epsilon t^{-1/2}\quad\text{as }t\to-\infty,$$
where $u^{(sol)}_-(x,t)=\overline{\check{u}^{(sol)}_+}(x,-t)$. Making again use of the symmetry of (\ref{equ: vx=Pv}), we know that $u^{(sol)}_-$ admits scattering data $(0;z'_1,...,z'_N;c^-_1,...,c^-_N)$ where
\begin{equation*}
    c^-_j = -\overline{\check{c}^+_j} \\
   = c_j'(\overline{\check{\Lambda}_{j}^{+}})^2
   = c_j'\exp\left(\frac{1}{\pi\ii} \int_{\re(z_{j})}^{\infty} \frac{\log(1+|r(\varsigma))|^2} {\varsigma-z_{j}'}d\varsigma
          \right).
\end{equation*}
The latter equality can be obtained easily from (\ref{equ: check c^+_j}) and shows us that (\ref{equ: def Lambda}) is true. Thus the proof of Theorem \ref{thm: mainthm} is completed.
\begin{rem}\label{rem: distinct ground states}
    The two ground states $u^{(sol)}_{\pm}$ are in general distinct which follows immediately from the distinct expressions for $\Lambda_{j}^+$ and $\Lambda_{j}^-$, respectively (see (\ref{equ: def Lambda})).
\end{rem} 

\bibliographystyle{alpha}
\bibliography{lit}

\end{document}